\documentclass[11pt,reqno]{amsart}
\usepackage{amsmath,amsthm,amssymb}

\newtheorem{theorem}{Theorem}[section]
\newtheorem{lemma}[theorem]{Lemma}
\newtheorem{proposition}[theorem]{Proposition}
\newtheorem{corollary}[theorem]{Corollary}

\theoremstyle{definition}

\newtheorem{definition}[theorem]{Definition}

\newtheorem{remark}[theorem]{Remark}
\newtheorem{notation}[theorem]{Notation}
\newtheorem{example}[theorem]{Example}

\numberwithin{equation}{section}

\begin{document}

\title{$\frak{g}$-quasi-Frobenius Lie algebras}

\author{David N. Pham}
\address{Department of Mathematics $\&$ Computer Science, QCC CUNY, Bayside, NY 11364}
\curraddr{}
\email{dpham90@gmail.com}
\thanks{This work is supported by PSC-CUNY Research Award $\#$ 69041-0047.}
\thanks{The author wishes to thank J. Funk and F. Ye for helpful discussions.}

\subjclass[2010]{22Exx, 22E60, 53D05, 18A05, 18E05}

\keywords{symplectic Lie groups, quasi-Frobenius Lie algebras, Lie bialgebras, Drinfeld double, group actions}

\dedicatory{}

\begin{abstract}
A Lie version of Turaev's $\overline{G}$-Frobenius algebras from 2-dimensional homotopy quantum field theory is proposed.  The foundation for this Lie version is a structure we call a \textit{$\frak{g}$-quasi-Frobenius Lie algebra} for $\frak{g}$ a finite dimensional Lie algebra.  The latter consists of a quasi-Frobenius Lie algebra $(\frak{q},\beta)$ together with a left $\frak{g}$-module structure which acts on $\frak{q}$ via derivations and for which $\beta$ is $\frak{g}$-invariant.   Geometrically, $\frak{g}$-quasi-Frobenius Lie algebras are the Lie algebra structures associated to symplectic Lie groups with an action by a Lie group $G$ which acts via symplectic Lie group automorphisms.  In addition to geometry, $\frak{g}$-quasi-Frobenius Lie algebras can also be motivated from the point of view of category theory.  Specifically, $\frak{g}$-quasi Frobenius Lie algebras correspond to \textit{quasi Frobenius Lie objects} in $\mathbf{Rep}(\frak{g})$.  If $\frak{g}$ is now equipped with a Lie bialgebra structure, then the categorical formulation of $\overline{G}$-Frobenius algebras given in \cite{KP} suggests that the Lie version of a $\overline{G}$-Frobenius algebra is a quasi-Frobenius Lie object in $\mathbf{Rep}(D(\frak{g}))$, where $D(\frak{g})$ is the associated (semiclassical) Drinfeld double.  We show that if $\frak{g}$ is a quasitriangular Lie bialgebra, then every $\frak{g}$-quasi-Frobenius Lie algebra has an induced $D(\frak{g})$-action which gives it the structure of a $D(\frak{g})$-quasi-Frobenius Lie algebra. 
\end{abstract}

\date{\today}

\maketitle

\section{Introduction}
Renewed interest in Frobenius algebras arose shortly after Witten's introduction of \textit{Topological Qunatum Field Theory} (TQFT) in \cite{Witt}.  Shortly afterwards, Atiyah proposed a set of axioms for TQFT \cite{At}, thus making Witten's work more accessible to the mathematical community.  Working from Atiyah's axioms, L. Abrams showed that 2-dimensional TQFTs are classified by commutative Frobenius algebras \cite{Ab}.  Hence, in the 2-dimensional case, the algebraic structure of a TQFT is that of a Frobenius algebra.  

The notion of a $(d+1)$-dimensional TQFT was generalized to a $(d+1)$-dimensional \textit{Homotopy Quantum Field Theory} (HQFT) by  V. Turaev in \cite{Tu} by equipping closed $d$-manifolds and $(d+1)$-dimensional cobordisms with homotopy classes of maps into a target space $X$.  In the special case when $X$ is a $K(\overline{G},1)$-space for $\overline{G}$ a finite group, one finds that the 2-dimensional HQFTs are classified by Frobenius algebras with a $\overline{G}$-grading and a $\overline{G}$-action which satisfies a number of conditions \cite{Tu, K}.   These Frobenius algebras came to be called \textit{$\overline{G}$-Frobenius algebras} (or \textit{crossed $\overline{G}$-algebras}).   

In \cite{KP}, a categorical formulation of $\overline{G}$-Frobenius algebras was presented where $\overline{G}$-Frobenius algebras were shown to correspond to certain types of Frobenius objects in $\mathbf{Rep}(D(k[\overline{G}]))$, the braided monoidal category of finite dimensional left $D(k[\overline{G}])$-modules, where $D(k[\overline{G}])$ is the Drinfeld double of the group ring $k[\overline{G}]$ with its usual Hopf structure.  Now the semiclassical analogue of $D(k[\overline{G}])$ (or more generally $D(H)$ for $H$ a finite dimensional Hopf algebra) is $D(\frak{g})$, the Drinfeld double of a finite dimensional Lie bialgbera $(\frak{g},\gamma)$  \cite{Drin, KS, CP, ES}.   The relationship between $\overline{G}$-Frobenius algebras and $D(k[\overline{G}])$ in \cite{KP} motivates the following question:\\

\noindent \textit{With $(\frak{g},\gamma)$ fixed, what structure plays the role of a $\overline{G}$-Frobenius algebra for $D(\frak{g})$?}\\  

Since $D(\frak{g})$ is the Lie version of $D(k[\overline{G}])$, the structure in question should be the Lie version of a $\overline{G}$-Frobenius algebra.  To answer the aforementioned question, we introduce the notion of \textit{$\frak{g}$-quasi-Frobenius Lie algebras} for $\frak{g}$ a finite dimensional Lie algebra. A $\frak{g}$-quasi-Frobenius Lie algebra consists of a quasi-Frobenius Lie algebra $(\frak{q},\beta)$ together with a left $\frak{g}$-module structure which acts on $\frak{q}$ via derivations and for which $\beta$ is $\frak{g}$-invariant. Geometrically,  $\frak{g}$-quasi-Frobenius Lie algebras are the Lie algebra structures of symplectic Lie groups with an action by a Lie group $G$ which acts via symplectic Lie group automorphisms.   We call the aforementioned structures \textit{$G$-symplectic Lie groups}.  

Interestingly,  $\frak{g}$-quasi-Frobenius Lie algebras have a categorical formulation.  To obtain this formulation, we introduce the notion of a \textit{quasi-Frobenius Lie object} for any additive symmetric monoidal category.  The work of Goyvaerts and Vercuysse on the categorification of Lie algebras \cite{GV} provides the foundation for defining quasi-Frobenius Lie objects.  The latter then yields an alternate (yet equivalent) definition of a $\frak{g}$-quasi-Frobenius Lie algebra: \textit{a $\frak{g}$-quasi Frobenius Lie algebra is simply a quasi Frobenius Lie object in $\mathbf{Rep}(\frak{g})$}, where $\mathbf{Rep}(\frak{g})$ is the category of finite dimensional representations of $\frak{g}$.  Using the categorical formulation of \cite{KP} as motivation, we obtain the Lie version of a $\overline{G}$-Frobenius algebra: for a fixed finite dimensional Lie bialgebra $(\frak{g},\gamma)$, the Lie version of a $\overline{G}$-Frobenius algebra is a quasi-Frobenius Lie object in $\mathbf{Rep}(D(\frak{g}))$.  In other words, with respect to $(\frak{g},\gamma)$, a $D(\frak{g})$-quasi-Frobenius Lie algebra is the Lie version of a $\overline{G}$-Frobenius algebra.  The definition of $D(\frak{g})$ implies that a $D(\frak{g})$-quasi-Frobenius Lie algebra is equivalent to a quasi-Frobenius Lie algebra $(\frak{q},\beta)$ which is both a $\frak{g}$ and $\frak{g}^\ast$-quasi-Frobenius Lie algebra where the $\frak{g}$ and $\frak{g}^\ast$ actions satisfy a  certain compatibility condition.  

The rest of the paper is organized as follows.  In section 2, we give a brief review of quasi-Frobenius Lie algebras, symplectic Lie groups, Lie bialgberas, and the Drinfeld double. In section 3, we formally define $\frak{g}$-quasi-Frobenius Lie algebras  and prove a general result for their construction.  We conclude the section with the categorical formulation of these structures.  In section 4, $G$-symplectic Lie groups are introduced.  We show that $\frak{g}$-quasi-Frobenius Lie algebras are the Lie algebra structures of $G$-symplectic Lie groups.  In addition, we show that the category of finite dimensional $\frak{g}$-quasi-Frobenius Lie algebras is equivalent to the category of simply connected $G$-symplectic Lie groups where $G$ is also simply connected.  In section 5, we focus our attention on $D(\frak{g})$-quasi-Frobenius Lie algebras.  We show that if $\frak{g}$ is a quasitriangular Lie bialgebra, then every $\frak{g}$-quasi-Frobenius Lie algebra has an induced $D(\frak{g})$-action which extends the original $\frak{g}$-action and gives the underlying quasi-Frobenius Lie algebra the structure of a $D(\frak{g})$-quasi-Frobenius Lie algebra.  In particular, for any finite dimensional Lie algebra $\frak{g}$ (viewed as a Lie bialgebra with co-bracket $\gamma\equiv 0$), every $\frak{g}$-quasi-Frobenius Lie algebra is a $D(\frak{g})$-quasi-Frobenius Lie algebra, where $D(\frak{g})$ is the Drinfeld double of $(\frak{g},0)$.

\section{Preliminaries}
\noindent In this section, we briefly review some of the relevant background for the current paper.  Throughout this section, ${k}$ is a field of characteristic zero.  

\subsection{Quasi-Frobenius Lie Algebras}
\noindent The definition of a \textit{Frobenius Lie algebra} \cite{Ooms, Ooms1} is modeled after the definition of a Frobenius algebra.  Formally, a Frobenius Lie algebra is defined as follows:
\begin{definition}
\label{QFLA1}
A \textit{Frobenius Lie algebra} over $k$ is a pair $(\frak{g},\alpha)$ where $\mathfrak{g}$ is a Lie algebra and $\alpha: \frak{g}\rightarrow k$ is a linear map with the property that the skew-symmetric bilinear form $\beta$ on $\frak{g}$ defined by
\begin{equation}
\nonumber
\beta(x,y):=\alpha([x,y])\hspace*{0.1in} \forall~x,y\in \frak{g}
\end{equation} 
is nondegenerate.
\end{definition}
\noindent As a consequence of the Jacobi identity, the skew-symmetric bilinear form $\beta$ in Definition \ref{QFLA1} satisfies the following identity: 
\begin{equation}
\label{QFLA2}
\beta([x,y],z)+\beta([y,z],x)+\beta([z,x],y)=0,\hspace*{0.1in} \forall~x,y,z\in \frak{g}.
\end{equation}
Equation (\ref{QFLA2}) is equivalent to the statement that $\beta$ is a 2-cocycle in the Lie algebra cohomology of $\frak{g}$ with values in $k$ (where $\frak{g}$ acts trivially on $k$).  This motivates the following generalization of Definition \ref{QFLA1}:
\begin{definition}
\label{QFLA3}
A \textit{quasi-Frobenius Lie algebra} over $k$ is a pair $(\frak{g},\beta)$ where $\frak{g}$ is a Lie algebra over $k$ and $\beta$ is a nondegenerate 2-cocycle in the Lie algebra cohomology of $\frak{g}$ with values in $k$ (where $\frak{g}$ acts trivially on $k$).
\end{definition}
\begin{remark}
\label{QFLA4}
A quasi-Frobenius Lie algebra $(\frak{g},\beta)$ is a Frobenius Lie algebra iff $\beta$ is exact, i.e., $\beta(x,y)=(-\delta \alpha)(x,y):=\alpha([x,y])$ for some linear map $\alpha: \frak{g}\rightarrow k$.  
\end{remark}

\begin{proposition}
\label{QFLA5}
Every 2-dimensional non-abelian Lie algebra admits the structure of a Frobenius Lie algebra.  In particular, every 2-dimensional non-abelian quasi-Frobenius Lie algebra is Frobenius.
\end{proposition}
\begin{proof}
Let $\frak{g}$ be a 2-dimensional non-abelian Lie algebra.  Then $\frak{g}$ admits a basis $u_1,~u_2$ such that $[u_1,u_2]=u_2$.  Let $\alpha: \frak{g}\rightarrow k$ be the linear map defined by $\alpha(u_1)=0$ and $\alpha(u_2)=1$.  Then $(\frak{g},\alpha)$ is a Frobenius Lie algebra.  

If $(\frak{g},\beta)$ is a quasi-Frobenius Lie algebra, set $\alpha(u_1)=0$ and $\alpha(u_2)=\beta(u_1,u_2)$.  Then it's easy to see that  $\beta(x,y)=\alpha([x,y])$ for all $x,y\in \frak{g}$.  Hence, $(\frak{g},\beta)$ is Frobenius.  
\end{proof}
\begin{remark}
\label{QFLA6}
Since every finite dimensional quasi-Frobenius Lie algebra $(\frak{g},\beta)$ is also a symplectic vector space, it follows that  the dimension of $\frak{g}$ is necessarily even.
\end{remark}

\begin{proposition}
\label{QLFA7}
Let $\frak{g}$ be a Lie algebra of dimension $n$ over $k$ and let $e_1,e_2,\dots, e_n$ be a basis of $\frak{g}$.  Then the following statements are equivalent: 
\begin{itemize}
\item[(1)] There exists $\alpha\in \frak{g}^\ast$ such that $(\frak{g},\alpha)$ is a Frobenius Lie algebra.
\item[(2)] There exists $\alpha\in \frak{g}^\ast$ such that $\mbox{det}(\alpha([e_i,e_j]))\neq 0$
\item[(3)] $\mbox{det}([e_i,e_j])\neq 0$, where $[e_i,e_j]\in \frak{g}$ are viewed as elements of the symmetric algebra $S(\frak{g})$.
\end{itemize}
\end{proposition}
\begin{proof}
$(1)\Leftrightarrow (2)$. Immediate.

$(2)\Rightarrow (3)$.   Recall that $S(\frak{g})$ is naturally isomorphic to the polynomial ring in $n$-variables where the variables are taken to be the basis $e_1,e_2,\dots, e_n$.  Extend the linear map $\alpha: \frak{g}\rightarrow k$ to a unit preserving algebra map $\alpha: S(\frak{g})\rightarrow k$ via
\begin{equation}
\nonumber
\alpha(v_1v_2\cdots v_r):=\alpha(v_1)\alpha(v_2)\cdots \alpha(v_r)
\end{equation}
for $v_1,\dots, v_r\in \frak{g}$.  Then 
\begin{align}
\nonumber
\alpha(\mbox{det}([e_i,e_j]))=\mbox{det}(\alpha([e_i,e_j]))\neq 0,
\end{align}
which implies that $\mbox{det}([e_i,e_j])\neq 0$.

$(2)\Leftarrow (3)$.  Let $p=\mbox{det}([e_i,e_j])\in S(\frak{g})$.  Since $p=p(e_1,\dots, e_n)\neq 0$ and $k$ is infinite, there exists $\lambda_i\in k$ such that $p(\lambda_1,\dots, \lambda_n)\neq 0$ (see Theorem 3.76 of \cite{Vin}).   Let $\alpha: \frak{g}\rightarrow k$ be the linear map defined by $\alpha(e_i)=\lambda_i$ for $i=1,\dots, n$.  As before, extend $\alpha: \frak{g}\rightarrow k$ to an algebra map $\alpha: S(\frak{g})\rightarrow k$.  Then 
\begin{align}
\nonumber
\mbox{det}(\alpha([e_i,e_j]))&=\alpha(\mbox{det}([e_i,e_j]))\\
\nonumber
&=\alpha(p(e_1,\dots,e_n))\\
\nonumber
&=p(\alpha(e_1),\dots, \alpha(e_n))\\
\nonumber
&=p(\lambda_1,\dots,\lambda_n)\\
\nonumber
&\neq 0.
\end{align}
\end{proof}

\noindent We now recall two examples.  The first is Frobenius and the second is quasi-Frobenius but not Frobenius \cite{Ooms1, Burde}.

\begin{example}
\label{QFLA8}
Let $\frak{g}$ be the 4-dimensional Lie algebra with basis $\{x_1,\dots, x_4\}$ and non-zero commutator relations:
\begin{align}
\nonumber
[x_1,x_2]=\frac{1}{2}x_2+x_3,\hspace*{0.1in} [x_1,x_3]=\frac{1}{2}x_3,\hspace*{0.1in}[x_1,x_4]=x_4,\hspace*{0.1in} [x_2,x_3]=x_4
\end{align}
Then $\mbox{det}([x_i,x_j])=(x_4)^4\neq 0$, where $[x_i,x_j]$ are regarded as elements of the symmetric algebra $S(\frak{g})$.  By Proposition \ref{QLFA7}, there exists a linear map $\alpha: \frak{g}\rightarrow k$ for which $(\frak{g},\alpha)$ is a Frobenius Lie algebra.
\end{example}

\begin{example}
\label{QFLA8a}
Let $\frak{q}$ be the 4-dimensional Lie algebra  with basis $\{x_1,\dots, x_4\}$ and non-zero commutator relations:
\begin{align}
\nonumber
[x_1,x_2]=x_3,\hspace*{0.1in} [x_1,x_3]=x_4
\end{align}
Since $\mbox{det}([x_i,x_j])=0$, $\frak{q}$ cannot be Frobenius by Proposition \ref{QLFA7}.  However, it does admit the structure of a quasi-Frobenius Lie algebra.  As an example of this, let $\beta$ be the nondegenerate, skew-symmetric bilinear form given by
\begin{equation}
\nonumber
\beta= x_1^\ast\wedge x_4^\ast+x_2^\ast\wedge x_3^\ast
\end{equation}
where $\{x_1^\ast,\dots, x_4^\ast\}$ is the dual basis.  A direct calculation shows that $\beta$ satisfies the 2-cocycle condition. Hence, $(\frak{q},\beta)$ is quasi-Frobenius.
\end{example}

\begin{definition}
\label{QFLA49}
Let $(\frak{g}_1,\beta_1)$ and $(\frak{g}_2,\beta_2)$ be quasi-Frobenius Lie algebras.  A \textit{quasi-Frobenius Lie algebra homomorphism} from $(\frak{g}_1,\beta_1)$  to $(\frak{g}_2,\beta_2)$  is a Lie algebra homomorphism $\varphi: \frak{g}_1\rightarrow \frak{g}_2$ such that $\varphi^\ast\beta_2=\beta_1$, that is,
\begin{equation}
\label{QFLA9e1}
\beta_1(u,v)=\beta_2(\varphi(u),\varphi(v)),\hspace*{0.1in} \forall~u,v\in \frak{g}_1.
\end{equation}
If $\varphi: \frak{g}_1\rightarrow \frak{g}_2$ satisfies (\ref{QFLA9e1}) and is also a Lie algebra isomorphism, then $\varphi$ is an isomorphism of quasi-Frobenius Lie algebras. 
\end{definition}

\begin{proposition}
\label{QFLA410}
Let $\varphi: (\frak{g}_1,\beta_1)\rightarrow (\frak{g}_2,\beta_2)$ be a quasi-Frobenius Lie algebra map.  If $\dim \frak{g}_1=\dim\frak{g}_2<\infty$, then $\varphi$ is an isomorphism of quasi-Frobenius Lie algebras. 
\end{proposition}
\begin{proof}
Since $\dim \frak{g}_1=\dim\frak{g}_2<\infty$, it suffices to show that $\varphi$ is injective.  Let $u\in \frak{g}_1$ be any nonzero element.  Since $\beta$ is nondegenerate, there exists $v\in \frak{g}_1$ such that $\beta(u,v)\neq 0$.  Hence, 
\begin{equation}
\nonumber
\beta_2(\varphi(u),\varphi(v))=\beta_1(u,v)\neq 0,
\end{equation}
which implies that $\varphi(u)\neq 0$.  This completes the proof.  
\end{proof}

\subsection{Symplectic Lie Groups}
\noindent In this section, we recall the correspondence between \textit{symplectic Lie groups} \cite{BC, Chu} and quasi-Frobenius Lie algebras.   
\begin{definition}
\label{SLG1}
A \textit{symplectic Lie group} is a pair $(G,\omega)$ where $G$ is a Lie group and $\omega$ is a left-invariant symplectic form on $G$.  
\end{definition}
\noindent The next result shows that the Lie algebra of a symplectic Lie group is naturally a quasi-Frobenius Lie algebra.  
\begin{proposition}
\label{SLG2}
Let $(G,\omega)$ be a symplectic Lie group.  Then $(\frak{g},\omega_e)$ is a quasi-Frobenius Lie algebra.
\end{proposition}
\begin{proof}
Let $\mathfrak{X}_l(G)$ denote the space of left-invariant vector fields on $G$ and endow $\frak{g}:=T_eG$ with the Lie algebra structure of $\mathfrak{X}_l(G)$.   Also, let $\tilde{x}$ denote the left-invariant vector field associated with $x\in \frak{g}$.  We now show that $(\frak{g},\omega_e)$ is a quasi-Frobenius Lie algebra.  Since $\omega_g|_{T_gG}$ is nondegenerate for all $g\in G$ (in particular for $g=e$), it only remains to show that $\omega_e$ is a 2-cocycle of $\frak{g}$ with values in $\mathbb{R}$ (where $\frak{g}$ acts trivially on $\mathbb{R}$).  

First, note that for any $x,y\in \frak{g}$, $\omega(\tilde{x},\tilde{y})$ is a constant function on $G$.  Indeed, for $g\in G$
\begin{align}
\nonumber
(\omega(\tilde{x},\tilde{y}))(g)&:=\omega_g(\tilde{x}_g,\tilde{y}_g)\\
\nonumber
&=\omega_g((l_g)_\ast x,(l_g)_\ast y)\\
\nonumber
&=(l_g^\ast \omega)_e(x,y)\\
\nonumber
&=\omega_e(x,y)
\end{align}
where the last equality follows from the fact that $\omega$ is left-invariant.   This fact along with the fact the $\omega$ is closed implies that $\omega_e\in Z^2(\frak{g};\mathbb{R})$:
\begin{align}
\nonumber
0&=d\omega(\tilde{x},\tilde{y},\tilde{z})\\
\nonumber
&=\tilde{x}(\omega(\tilde{y},\tilde{z}))-\tilde{y}(\omega(\tilde{x},\tilde{z}))+\tilde{z}(\omega(\tilde{x},\tilde{y}))\\
\nonumber
&-\omega([\tilde{x},\tilde{y}],\tilde{z})+\omega([\tilde{x},\tilde{z}],\tilde{y})-\omega([\tilde{y},\tilde{z}],\tilde{x})\\
\nonumber
&=-\omega([\tilde{x},\tilde{y}],\tilde{z})-\omega([\tilde{z},\tilde{x}],\tilde{y})-\omega([\tilde{y},\tilde{z}],\tilde{x}).
\end{align}
Evaluating the last equality at $e\in G$ and multiplying by $-1$ gives the 2-cocycle condition on $\omega_e$:
\begin{equation}
\nonumber
\omega_e([{x},{y}],{z})+\omega_e([{z},{x}],{y})+\omega_e([{y},{z}],{x})=0.
\end{equation}
Hence, $(\frak{g},\omega_e)$ is a quasi-Frobenius Lie algebra.
\end{proof}

\begin{proposition}
\label{SLG3}
Let $G$ be a Lie group whose Lie algebra $\frak{g}$ carries the structure of a quasi-Frobenius Lie algebra with 2-cocycle $\beta$.  Define $\widetilde{\beta}\in \Omega^2(G)$ by
\begin{equation}
\nonumber
\widetilde{\beta}_g:=(l_{g^{-1}})^\ast \beta\in \wedge^2 T^\ast_g G,\hspace*{0.2in}\forall ~g\in G
\end{equation}
where  $l_g: G\rightarrow G$ is left translation by $g$.  Then $(G,\widetilde{\beta})$ is a symplectic Lie group whose associated quasi-Frobenius Lie algebra is $(\frak{g},\widetilde{\beta}_e)=(\frak{g},\beta)$.
\end{proposition}
\begin{proof}
It follows immediately from the definition that $\widetilde{\beta}$ is left-invariant, that is, $(l_g)^\ast \widetilde{\beta}=\widetilde{\beta}$ for all $g\in G$.  Moreover, since $\beta$ is nondegenerate, $\widetilde{\beta}$ must be nondegenerate as well. To see that $d\widetilde{\beta}=0$, it suffices to show that $d\widetilde{\beta}(\tilde{x},\tilde{y},\tilde{z})=0$ for all left-invariant vector fields $\tilde{x}$, $\tilde{y}$, and $\tilde{z}$.  Since $\widetilde{\beta}$ is left-invariant, it follows that $\widetilde{\beta}(\tilde{x},\tilde{y})=\widetilde{\beta}_e(x,y)=\beta(x,y)$ is a constant function on $G$ for all  left-invariant vector fields $\tilde{x}$ and $\tilde{y}$, where $\tilde{x}_e=x$ and $\tilde{y}_e=y$.  In particular,
\begin{equation}
\nonumber
\widetilde{\beta}([\tilde{x},\tilde{y}],\tilde{z})=\beta([x,y],z).
\end{equation}
The proof of Proposition \ref{SLG2} shows that if $\widetilde{\beta}$ is left-invariant, we have
\begin{align}
\nonumber
d\widetilde{\beta}(\tilde{x},\tilde{y},\tilde{z})&=-\widetilde{\beta}([\tilde{x},\tilde{y}],\tilde{z})-\widetilde{\beta}([\tilde{z},\tilde{x}],\tilde{y})-\widetilde{\beta}([\tilde{y},\tilde{z}],\tilde{x})\\
\nonumber
&=-\beta([{x},{y}],{z})-\beta([{z},{x}],{y})-\beta([{y},{z}],{x}).
\end{align}
Since $\beta\in Z^2(\frak{g};\mathbb{R})$, the last equality must be zero.  Hence, $(G,\widetilde{\beta})$ is a symplectic Lie group.
\end{proof}
\begin{definition}
\label{SLG4}
Let $(G,\omega)$ and $(H,\sigma)$ be symplectic Lie groups.  A homomorphism of symplectic Lie groups is a Lie group homomorphism $\varphi: G\rightarrow H$ such that $\varphi^\ast\sigma=\omega$.  
\end{definition}

\begin{lemma}
\label{SLG5}
Let $(G,\omega)$ and $(H,\sigma)$ be symplectic Lie groups and let $\varphi: G{\rightarrow} H$ be a Lie group homomorphism.  Then $\varphi^\ast\sigma =\omega$ iff $(\varphi^\ast\sigma)_e=\omega_e$.
\end{lemma}
\begin{proof}
$(\Rightarrow)$ Suppose $(\varphi^\ast\sigma)=\omega$.  By definition, $(\varphi^\ast\sigma)_g=\omega_g$ for all $g\in G$.  In particular, the equality holds for $g=e$.

$(\Leftarrow)$ Now suppose $(\varphi^\ast\sigma)_e=\omega_e$.  Let $g\in G$ and $x,y\in T_g G$.  Then
\begin{align}
\nonumber
(\varphi^\ast\sigma)_g(x,y)&=\sigma_{\varphi(g)}(\varphi_{\ast,g}(x),\varphi_{\ast,g}(y))\\
\nonumber
&=[(l_{\varphi(g^{-1})})^\ast \sigma_e](\varphi_{\ast,g}(x),\varphi_{\ast,g}(y))\\
\nonumber
&=\sigma_e((l_{\varphi(g^{-1})}\circ\varphi)_{\ast,g}(x),(l_{\varphi(g^{-1})}\circ\varphi)_{\ast,g}(y))\\
\nonumber
&=\sigma_e((\varphi\circ l_{g^{-1}})_{\ast,g}(x),(\varphi\circ l_{g^{-1}})_{\ast,g}(y))\\
\nonumber
&=(\varphi^\ast \sigma)_e(( l_{g^{-1}})_{\ast,g}(x),( l_{g^{-1}})_{\ast,g}(y))\\
\nonumber
&=\omega_e(( l_{g^{-1}})_{\ast,g}(x),( l_{g^{-1}})_{\ast,g}(y))\\
\nonumber
&=[(l_{g^{-1}})^\ast\omega_e](x,y)\\
\nonumber
&=\omega_g(x,y),
\end{align}
where the second and last equalities follow from the left-invariance of $\sigma$ and $\omega$ respectively and the fourth equality follows from the fact that $\varphi$ is a group homomorphism.  This completes the proof.
\end{proof}

\begin{proposition}
\label{SLG6}
Let $\varphi: (G,\omega)\rightarrow (H,\sigma)$ be a homomorphism of symplectic Lie groups.  Then 
\begin{equation}
\nonumber
\varphi_{\ast,e}: (\frak{g},\omega_e)\rightarrow (\frak{h},\sigma_e)
\end{equation}
is a homomorphism of quasi-Frobenius Lie algebras.
\end{proposition}
\begin{proof}
This follows immediately from the properties of $\varphi$.
\end{proof}

\begin{proposition}
\label{SLG6a}
Let $\psi: (\frak{g},\beta)\rightarrow (\frak{h},\sigma)$ be a homomorphism of quasi-Frobenius Lie algebras.  Let $G$ be the simply connected Lie group whose Lie algebra is $\frak{g}$ and let $H$ be any Lie group whose Lie algebra is $\frak{h}$.  Let $(G,\widetilde{\beta})$ and $(H,\widetilde{\sigma})$ be the symplectic Lie groups associated to $(\frak{g},\beta)$ and $(\frak{h},\sigma)$ respectively (see Proposition \ref{SLG3}).  Then there exists a unique symplectic Lie group homomorphism
\begin{equation}
\nonumber
\widehat{\psi}: (G,\widetilde{\beta})\rightarrow (H,\widetilde{\sigma})
\end{equation}
such that $\widehat{\psi}_{\ast,e}=\psi$.
\end{proposition}

\begin{proof}
Since $G$ is simply connected, there exists a unique Lie group homomorphism $\widehat{\psi}: G\rightarrow H$ such that $\widehat{\psi}_{\ast,e}=\psi$.  It only remains to show that $\widehat{\psi}^\ast \widetilde{\sigma}=\widetilde{\beta}$.  By Lemma \ref{SLG5}, it suffices to show that $(\widehat{\psi}^\ast \widetilde{\sigma})_e=\widetilde{\beta}_e=\beta$.  To do this, let $x,y\in \frak{g}$.  Then
\begin{align}
\nonumber
(\widehat{\psi}^\ast \widetilde{\sigma})_e(x,y)&=\widetilde{\sigma}_{\widehat{\psi}(e)}(\widehat{\psi}_{\ast,e}(x),\widehat{\psi}_{\ast,e}(y))\\
\nonumber
&=\widetilde{\sigma}_e(\psi(x),\psi(y))\\
\nonumber
&=\sigma(\psi(x),\psi(y))\\
\nonumber
&=(\psi^\ast\sigma)(x,y)\\
\nonumber
&=\beta(x,y).
\end{align}
This completes the proof.
\end{proof}

\begin{theorem}
\label{SLG7}
Let $\mathbf{SCSLG}$ be the category of simply connected symplectic Lie groups and let $\mathbf{qFLA}$ be the category of finite dimensional quasi-Frobenius Lie algberas.  Let $F$ be the functor from $\mathbf{SCSLG}$ to $\mathbf{qFLA}$ which sends $(G,\omega)$ to $(\frak{g},\omega_e)$ and $\varphi: (G,\omega)\rightarrow (H,\sigma)$ to $\varphi_{\ast,e}: (\frak{g},\omega_e)\rightarrow (\frak{h},\sigma_e)$.  Then $F$ is an equivalence of categories.  
\end{theorem}
\begin{proof}
Theorem \ref{SLG7} follows from the well known correspondence between simply connected Lie groups and finite dimensional Lie algebras combined with Proposition \ref{SLG2}, Proposition \ref{SLG3}, Proposition \ref{SLG6}, and Proposition \ref{SLG6a}.
\end{proof}

\noindent As an example, we now recall the symplectic Lie group structure on the affine Lie group $A(n,\mathbb{R})$ (c.f., \cite{Aga, Mik, Ooms1}).
\begin{example}
\label{SLG8}
Recall that $A(n,\mathbb{R})$ is the Lie group consisting of $(n+1)\times (n+1)$ matrices of the form
\begin{equation}
\nonumber
A(n,\mathbb{R})=\left\{\left(\begin{array}{ll}
A & v\\
0 & 1
\end{array}
\right)\big| ~A\in GL(n,\mathbb{R}),~v\in \mathbb{R}^n
\right\}.
\end{equation}
The associated Lie algebra is then
\begin{equation}
\nonumber
\frak{a}(n,\mathbb{R})=\left\{\left(\begin{array}{ll}
A & v\\
0 & 0
\end{array}
\right)\big| ~A\in \frak{gl}(n,\mathbb{R}),~v\in \mathbb{R}^n
\right\}.
\end{equation}
From the definition of $A(n,\mathbb{R})$, we see that $A(n,\mathbb{R})$ is even dimensional with $\dim ~A(n,\mathbb{R})=\dim \frak{a}(n,\mathbb{R})=n^2+n=n(n+1)$.  Let $E_{ij}$ denote the $(n+1)\times (n+1)$ matrix with $1$ in the $(i,j)$-component and all other components zero.  Then $\{E_{ij}\}_{1\le i\le n,~1\le j\le n+1}$ is a basis on $\frak{a}(n,\mathbb{R})$.  Let $\{E_{ij}^\ast\}_{1\le i\le n,~1\le j\le n+1}$ denote the corresponding dual basis.  Define
\begin{equation}
\nonumber
\alpha=E_{12}^\ast+E_{23}^\ast+\cdots + E_{n,n+1}^\ast
\end{equation}
and $\beta(X,Y):=-\delta\alpha(X,Y)=\alpha([X,Y])$ for all $X,Y\in \frak{a}(n,\mathbb{R})$.  Since
\begin{equation}
\nonumber
[E_{ij},E_{kl}]=\delta_{jk} E_{il}-\delta_{li} E_{kj},
\end{equation}
we see that 
\begin{equation}
\label{SLG8e1}
\beta(E_{ij},E_{kl})=\delta_{jk}\delta_{l,i+1}-\delta_{li}\delta_{j,k+1}.
\end{equation}
Careful consideration of (\ref{SLG8e1}) shows that $\beta:=-\delta\alpha\in Z^2(\frak{a}(n,\mathbb{R});\mathbb{R})$ is nondegenerate.  Hence, $(\frak{a}(n,\mathbb{R}),\alpha)$ is a Frobenius Lie algebra. (In particular, $(\frak{a}(n,\mathbb{R}),\beta)$ is a quasi-Frobenius Lie algebra.)  Let $\tilde{\beta}\in \Omega^2(A(n,\mathbb{R}))$ be the left-invariant 2-form on $A(n,\mathbb{R})$ associated to $\beta$.  Then $(A(n,\mathbb{R}),\widetilde{\beta})$ is a symplectic Lie group.  Furthermore, since $\beta:=-\delta\alpha$, it follows that $\tilde{\beta}$ is exact.  Specifically, 
\begin{equation}
\nonumber
\widetilde{\beta}=-d\tilde{\alpha}
\end{equation}
where $\tilde{\alpha}\in \Omega^1(A(n,\mathbb{R}))$ is the left-invariant 1-form on $A(n,\mathbb{R})$ associated to $\alpha$.
\end{example}

\subsection{Lie bialgebras $\&$ the Drinfeld Double}

\begin{definition}
\label{SSD1}
A \textit{Lie bialgebra} over a field $k$ is a pair $(\frak{g},\gamma)$ where $\frak{g}$ is a Lie algebra over $k$ and $\gamma:\frak{g}\rightarrow  \frak{g}\wedge \frak{g}\subset \frak{g}\otimes \frak{g}$ is a skew-symmetric linear map such that
\begin{itemize}
\item[1.] $\gamma^\ast: \frak{g}^\ast\otimes \frak{g}^\ast\rightarrow \frak{g}^\ast$ is a Lie bracket on $\frak{g}^\ast$, where the dual map $\gamma^\ast$ is resrticted to $\frak{g}^\ast\otimes \frak{g}^\ast\subset (\frak{g}\otimes \frak{g})^\ast$;
\item[2.] $\gamma$ is a 1-cocycle on $\frak{g}$ with values in $\frak{g}\otimes \frak{g}$, where $\frak{g}$ acts on $\frak{g}\otimes \frak{g}$ via the adjoint action.
\end{itemize}
$\gamma$ is called the \textit{cobracket} or \textit{co-commutator}.
\end{definition}
\noindent Condition 2 in Definition \ref{SSD1} is equivalent to the condition 
\begin{equation}
\nonumber
\gamma([x,y])=ad_x^{(2)}\gamma(y)-ad_y^{(2)}\gamma(x),~\forall~x,y\in\frak{g}
\end{equation}
where the linear map $ad_x^{(2)}: \frak{g}\otimes \frak{g}\rightarrow \frak{g}\otimes \frak{g}$ is the adjoint action of $x\in \frak{g}$ on $\frak{g}\otimes \frak{g}$.  Explicitly, $ad_x^{(2)}$ is defined via
\begin{equation}
\nonumber
ad_x^{(2)}(y\otimes z)=ad_x(y)\otimes z+y\otimes ad_x(z)=[x,y]\otimes z+y\otimes [x,z]
\end{equation}
for $y,z\in \frak{g}$.

\begin{definition}
\label{SSD2}
Let $(\frak{g},\gamma_{\frak{g}})$ and $(\frak{h},\gamma_{\frak{h}})$ be Lie bialgebras.  A \textit{Lie bialgebra homomorphism} from $(\frak{g},\gamma_{\frak{g}})$ to $(\frak{h},\gamma_{\frak{h}})$ is a Lie algebra map $\varphi: \frak{g}\rightarrow \frak{h}$ such that
\begin{equation}
\nonumber
(\varphi\otimes \varphi)\circ \gamma_{\frak{g}}=\gamma_{\frak{h}}\circ \varphi.
\end{equation}
\end{definition}

\begin{example}
\label{SSD3}
Any Lie algebra $\frak{g}$ can be turned into a Lie bialgebra by taking the cobracket $\gamma\equiv 0$.  $(\frak{g},0)$ is the \textit{trivial} Lie bialgebra structure on $\frak{g}$.  
\end{example}

\noindent The next result shows that the notion of a Lie bialgebra is self-dual for the finite dimensional case.
\begin{proposition}
\label{SSD5}
Let $(\frak{g},\gamma_{\frak{g}})$ be a finite dimensional Lie bialgebra and let $\gamma_{\frak{g}^\ast}: \frak{g}^\ast\rightarrow \frak{g}^\ast\otimes \frak{g}^\ast$ be the dual of the Lie bracket on $\frak{g}$.  Then $(\frak{g}^\ast,\gamma_{\frak{g}^\ast})$ is a Lie bialgebra, where the Lie bracket on $\frak{g}^\ast$ is given by the dual of $\gamma_{\frak{g}}$.     
\end{proposition}

\noindent For a Lie algebra $\frak{g}$, the simplest way to obtain an element of $Z^1_{ad}(\frak{g};\frak{g}\otimes \frak{g})$ is to turn to the 0-cochains and take their coboundaries.  This raises the following natural question: given $r\in \frak{g}\otimes \frak{g}$, when does $\delta r\in  Z^1_{ad}(\frak{g};\frak{g}\otimes \frak{g})$ define a Lie bialgebra structure on $\frak{g}$?  To answer this question, let
\begin{equation}
\nonumber
r=\sum_i a_i\otimes b_i,
\end{equation}
and define
\begin{equation}
\label{SSD6}
[[r,r]]:=[r_{12},r_{13}]+[r_{12},r_{23}]+[r_{13},r_{23}],
\end{equation}
where
\begin{align}
\label{SSD7}
[r_{12},r_{13}]&:=\sum_{i,j}[a_i,a_j]\otimes b_i\otimes b_j\\
\label{SSD8}
[r_{12},r_{23}]&:=\sum_{i,j}a_i\otimes [b_i,a_j]\otimes b_j\\
\label{SSD9}
[r_{13},r_{23}]&:=\sum_{i,j}=a_i\otimes a_j\otimes [b_i,b_j].
\end{align}

\begin{definition}
\label{SSD10}
A \textit{coboundary Lie bialgebra} is a Lie bialgebra $(\frak{g},\gamma)$ such that $\gamma=\delta r$ for some $r\in \frak{g}\otimes \frak{g}$.  The element $r$ is called the \textit{$r$-matrix}.
\end{definition}

\noindent The next result provides a necessary and sufficient condition for an element $r\in \frak{g}\otimes \frak{g}$ to define a Lie bialgebra structure on $\frak{g}$.

\begin{proposition}
\label{SSD11}
Let $\frak{g}$ be a Lie algebra.  Then $(\frak{g},\delta r)$ is a Lie bialgebra iff 
\begin{itemize}
\item[(i)] $r+\sigma(r)$ is invariant under the adjoint action of $\frak{g}$ on $\frak{g}\otimes \frak{g}$, where $\sigma: \frak{g}\otimes \frak{g}\rightarrow \frak{g}\otimes \frak{g}$ is the unique linear map defined by $x\otimes y\mapsto y\otimes x$ for $x,y\in \frak{g}$;
\item[(ii)] $[[r,r]]$ is invariant under the adjoint action of $\frak{g}$ on $\frak{g}\otimes \frak{g}\otimes \frak{g}$.
\end{itemize}
\end{proposition}
\begin{proof}
See pp. 51-54 of \cite{CP}.
\end{proof}
\noindent The simplest way to ensure that condition (ii) of Proposition \ref{SSD11} is satisfied is to demand that 
\begin{equation}
\label{SSD12}
[[r,r]]=0.
\end{equation}
Equation \ref{SSD12} is called the \textit{classical Yang-Baxter equation} (CYBE).  The CYBE motivates the following definition:
\begin{definition}
\label{SSD13}
A coboundary Lie bialgebra $(\frak{g},\delta r)$ is \textit{quasitriangular} if $r$ is a solution of the CYBE.   Furthermore, if $r$ is skew-symmetric, that is, $r\in \frak{g}\wedge\frak{g}\subset \frak{g}\otimes \frak{g}$, then $(\frak{g},\delta r)$ is said to be \textit{triangular}.  
\end{definition}

\begin{example}
\label{trLBA}
Let $\frak{g}$ be the two dimensional Lie algebra with basis $x,y$ and commutator relation $[x,y]=x$.  Define $r=y\wedge x$.  Then $(\frak{g},\delta r)$ is a triangular Lie bialgebra, where $\gamma:=\delta r$ is given explicitly by
\begin{equation}
\nonumber
\gamma(x)=0,\hspace*{0.2in} \gamma(y)=x\wedge y.
\end{equation}
\end{example}
\noindent Before turning to the Drinfeld double, we recall the following notion:
\begin{definition}
Let $\frak{g}$ be a Lie algebra and let $\langle\cdot ,\cdot\rangle$ be  a bilinear form on $\frak{g}$.  $\frak{g}$ is \textit{ad-invariant} with respect to $\langle\cdot ,\cdot\rangle$ if 
\begin{equation}
\langle [x,y],z\rangle = \langle x,[y,z]\rangle,\hspace*{0.1in} \forall ~x,y,z\in \frak{g}.
\end{equation}
\end{definition}
\noindent  Now let $(\frak{g},\gamma_{\frak{g}})$ be a finite dimensional Lie bialgebra and let $(\frak{g}^\ast,\gamma_{\frak{g}^\ast})$ be the associated dual Lie bialgebra.   Consider the direct sum
\begin{equation}
\nonumber
\frak{g}\oplus \frak{g}^\ast
\end{equation}
and equip it with the symmetric, nondegenerate bilinear form $\langle\cdot, \cdot\rangle$ defined by
\begin{equation}
\nonumber
\langle x+\xi,y+\eta\rangle = \xi(y)+\eta(x),
\end{equation}
where we write $x+\xi$ and $y+\eta$ for $(x,\xi),~(y,\eta)\in \frak{g}\oplus \frak{g}^\ast$.  The Drinfeld double of $(\frak{g},\gamma_{\frak{g}})$, denoted by $D(\frak{g})$, is the unique quasitriangular Lie bialgebra which satisfies the following condtions:
\begin{itemize}
\item[(1)] As a vector space, 
\begin{equation}
\nonumber
D(\frak{g})=\frak{g}\oplus \frak{g}^\ast.
\end{equation}
\item[(2)] As a Lie algebra, $D(\frak{g})$ is ad-invariant with respect  to the inner product $\langle\cdot,\cdot\rangle$ and contains $\frak{g}$ and $\frak{g}^\ast$ as Lie subalgebras.
\item[(3)] The cobracket on $D(\frak{g})$ is defined by $\gamma_D:=\gamma_{\frak{g}}-\gamma_{\frak{g}^\ast}$.
\end{itemize}
Let $[\cdot,\cdot]_D$, $[\cdot,\cdot]_{\frak{g}}$, and $[\cdot,\cdot]_{\frak{g}^\ast}$ denote the Lie brackets on $D(\frak{g})$, $\frak{g}$, and $\frak{g}^\ast$ respectively.  Condition (2) implies that 
\begin{equation}
\nonumber
[x,y]_D=[x,y]_{\frak{g}},\hspace*{0.1in} [\xi,\eta]_D=[\xi,\eta]_{\frak{g}^\ast},\hspace*{0.1in} [x,\xi]_D=ad^\ast_x \xi-ad^\ast_\xi x
\end{equation}
for all $x,y\in \frak{g}$ and $\xi,\eta\in \frak{g}^\ast$, where $ad^\ast$ denotes the coadjoint action of $\frak{g}$ on $\frak{g}^\ast$ and $\frak{g}^\ast$ on $\frak{g}$.  Explicitly, $ad^\ast_x: \frak{g}^\ast\rightarrow \frak{g}^\ast$ and  $ad^\ast_\xi: \frak{g}\rightarrow \frak{g}$ are defined by $ad^\ast_x:=-ad^t_x$ and $ad^\ast_\xi:=-ad^t_\xi$ where $ad^t_x$ and $ad^t_\xi$ are the ordinary duals of $ad_x: \frak{g}\rightarrow \frak{g}$ and $ad_\xi: \frak{g}^\ast \rightarrow \frak{g}^\ast$.  In dealing with the Drinfeld double, we will drop the ``$D$", ``$\frak{g}$", and ``$\frak{g}^\ast$" that appear as subscripts in the Lie brackets of $D(\frak{g})$, $\frak{g}$, and $\frak{g}^\ast$ respectively.  Condition (2) implies that the triple $(D(\frak{g}),\frak{g},\frak{g}^\ast)$ is a \textit{Manin triple} with respect to the inner product $\langle \cdot, \cdot\rangle$.  In fact, there is a one to one correspondence between finite dimensional Lie bialgebras and Manin triples (see \cite{CP}).  

Lastly, condition (3) implies that $D(\frak{g})$ is quasitriangular with r-matrix
\begin{equation}
\nonumber
r=\sum_i e_i\otimes e_i^\ast
\end{equation}
where $e_1,\dots, e_n$ is any basis on $\frak{g}$ and $e_1^\ast,\dots, e_n^\ast$ is the corresponding dual basis.

\begin{example}
\label{DrinExample}
Let $(\frak{g},\gamma)$ be the 2-dimensional Lie bialgebra with basis $x,y$ satisfying $[x,y]=x$ and cobracket $\gamma(x)=0$ and $\gamma(y)=x\wedge y$.  Let $x^\ast, y^\ast$ denote the corresponding dual basis.  The commutator relations on  $D(\frak{g})$ are 
\begin{align}
\nonumber
[x,y]&=x,\hspace*{0.1in} [x^\ast,y^\ast]=y^\ast,\hspace*{0.1in} [x,x^\ast]=-y^\ast,\hspace*{0.1in}[x,y^\ast]=0\\
\nonumber
& [y,x^\ast]=x^\ast+y,\hspace*{0.1in} [y,y^\ast]=-x
\end{align}
The r-matrix is $r=x\otimes x^\ast+y\otimes y^\ast$.
\end{example} 

\section{$\frak{g}$-quasi-Frobenius Lie Algebras}
\noindent We begin with the formal definition:
\begin{definition}
\label{gqFLA8}
A \textit{$\frak{g}$-quasi-Frobenius Lie algebra} is a triple $(\frak{q},\beta,\rho)$ such that $(\frak{q},\beta)$ is a quasi-Frobenius Lie algebra and $\rho: \frak{g}\rightarrow \frak{gl}(\frak{q})$, $x\mapsto \rho_x$ is a left $\frak{g}$-module structure on $\frak{q}$ such that 
\begin{itemize}
\item[(i)] $\rho_x$ is a derivation on $\frak{q}$ for all $x\in \frak{g}$
\item[(ii)] $\beta(\rho_x(u),v)+\beta(u,\rho_x(v))=0$ for all $x\in \frak{g}$, $u,v\in \frak{q}$ ($\frak{g}$-invariance)
\end{itemize}
\end{definition}

\noindent In this section, we prove a result for the general construction of $\frak{g}$-quasi-Frobenius Lie algebras.  Before doing so, we make the following observation:
\begin{proposition}
\label{gqFLA11}
Let $(\frak{q},\beta)$ be a quasi-Frobenius Lie algebra and let $\mbox{Aut}(\frak{q},\beta)$ be the automorphism group of $(\frak{q},\beta)$.  Then $\mbox{Aut}(\frak{q},\beta)$ is an embedded Lie subgroup of $GL(\frak{q})$.
\end{proposition}
\begin{proof}
As a set,  $\mbox{Aut}(\frak{q},\beta)=\mbox{Aut}(\frak{q})\cap\mbox{Sp}(\frak{q},\beta)$ where $\mbox{Aut}(\frak{q})$ is the group of automorphisms of the Lie algebra $\frak{q}$ and $\mbox{Sp}(\frak{q},\beta)$ is the group of linear symplectomorphisms of $(\frak{q},\beta)$, where the latter is regarded as a symplectic vector space.  Since $\mbox{Aut}(\frak{q})$ and $\mbox{Sp}(\frak{q},\beta)$ are both closed subgroups of $GL(\frak{q})$, each being the zero set of a collection of polynomials, $\mbox{Aut}(\frak{q},\beta)$ is also a closed subgroup of $GL(\frak{q})$.  By the closed subgroup theorem \cite{Wa}, $\mbox{Aut}(\frak{q},\beta)$ is an embedded Lie subgroup of $GL(\frak{q})$.
\end{proof}

\begin{proposition}
\label{gqFLA12}
Let $(\frak{q},\beta)$ be a quasi-Frobenius Lie algebra and let 
\begin{equation}
\nonumber
\rho: G\rightarrow \mbox{Aut}(\frak{q},\beta)\subset GL(\frak{q}),~g\mapsto \rho_g
\end{equation}
be a Lie group homomorphism.  Define 
\begin{equation}
\nonumber
\rho':=\rho_{\ast,e}: \frak{g}\rightarrow \frak{gl}(\frak{q}),~\hspace*{0.1in} x\mapsto \rho'_x.
\end{equation}
Then $(\frak{q},\beta,\rho')$ is a $\frak{g}$-quasi-Frobenius Lie algebra.  In particular, if $G$ is any Lie subgroup of $\mbox{Aut}(\frak{q},\beta)$, then $(\frak{q},\beta)$ admits the structure of a $\frak{g}$-quasi-Frobenius Lie algebra.
\end{proposition}

\begin{proof}
Since $\rho$ is a Lie group homomorphism, it immediately follows that $\rho': \frak{g}\rightarrow \frak{gl}(\frak{q})$ is a representation of $\frak{g}$ on $\frak{q}$.  We now show that
\begin{equation}
\label{gqFLA12e1}
\rho_x([u,v])=[\rho_x(u),v]+[u,\rho_x(v)]
\end{equation}
and
\begin{equation}
\label{gqFLA12e2}
\beta(\rho_x(u),v)+\beta(u,\rho_x(v))=0
\end{equation}
for all $x\in \frak{g}$ and $u,v\in \frak{q}$.  To do this,  fix a basis $e_1,e_2,\dots, e_n$ on $\frak{q}$.  Since $\rho_{\mbox{exp}(tx)}(u),~\rho_{\mbox{exp}(tx)}(v)\in \frak{q}$, we have
\begin{equation}
\label{gqFLA12e3}
\rho_{\mbox{exp}(tx)}(u)=\sum_i a_i(t)e_i,\hspace*{0.2in}\rho_{\mbox{exp}(tx)}(v)=\sum_i b_i(t)e_i
\end{equation}
for some smooth functions $a_i(t)$, $b_i(t)$, $i=1,\dots, n$.  Hence,  
\begin{equation}
\label{gqFLA12e4}
\rho'_{x}(u)=\sum_i \dot{a}_i(0)e_i,\hspace*{0.2in}\rho'_x(v)=\sum_i \dot{b}_i(0)e_i.
\end{equation}
Since $\rho_g\in \mbox{Aut}(\frak{q},\beta)$ for all $g\in G$, we have
\begin{align}
\label{gqFLA12e5}
\rho_{\mbox{exp}(tx)}([u,v])&=[\rho_{\mbox{exp}(tx)}(u),\rho_{\mbox{exp}(tx)}(v)].
\end{align}
Substituting (\ref{gqFLA12e3}) into the right side of (\ref{gqFLA12e5}) and applying $\frac{d}{dt}|_{t=0}$ to both sides of (\ref{gqFLA12e5}) gives
\begin{align}
\nonumber
\rho'_x([u,v])&=\frac{d}{dt}\big|_{t=0}[\rho_{\mbox{exp}(tx)}(u),\rho_{\mbox{exp}(tx)}(v)]\\
\nonumber
&=\frac{d}{dt}\big|_{t=0} \sum_{i,j} a_i(t)b_j(t)[e_i,e_j]\\
\nonumber
&=\sum_{i,j}(\dot{a}_i(0)b_j(0)[e_i,e_j]+a_i(0)\dot{b}_j(0)[e_i,e_j])\\
\label{gqFLA12e6}
&=[\rho'_x(u),v]+[u,\rho'_x(v)],
\end{align}
which proves (\ref{gqFLA12e1}).  

For equation (\ref{gqFLA12e2}), note that 
\begin{equation}
\label{gqFLA12e7}
\beta(\rho_{\mbox{exp}(tx)}(u),\rho_{\mbox{exp}(tx)}(v))=\beta(u,v)
\end{equation}
since $\rho_g\in \mbox{Aut}(\frak{q},\beta)$ for all $g\in G$.  Substituting (\ref{gqFLA12e3}) into the left side of (\ref{gqFLA12e7}) and applying $\frac{d}{dt}\big|_{t=0}$ to both sides of (\ref{gqFLA12e7}) gives
\begin{equation}
\nonumber
\beta(\rho'_x(u),v)+\beta(u,\rho'_x(v))=0.
\end{equation}
This completes the proof.
\end{proof}
\noindent A trivial example of a $\frak{g}$-quasi-Frobenius Lie algebra is obtained by equipping any quasi-Frobenius Lie algebra with the trivial $\frak{g}$-action.  We now consider a more interesting example which is an application of Proposition \ref{gqFLA12}.
\begin{example}
\label{gqFLA13}
Let $\frak{q}$ be the 4-dimensional Lie algebra $\{e_1,e_2,e_3,e_4\}$ whose non-zero commutator relations are given by  \cite{Burde}:
\begin{align}
\nonumber
[e_1,e_2]=e_2,\hspace*{0.1in} [e_1,e_3]=e_3,\hspace*{0.1in}[e_1,e_4]=2e_4,\hspace*{0.1in}[e_2,e_3]=e_4.
\end{align}
Let $\alpha:\frak{q}\rightarrow \mathbb{R}$ be the linear map defined by $\alpha(e_i)=0$ for $i=1,2,3$ and $\alpha(e_4)=1$.  Define $\beta(u,v):=\alpha([u,v])$ for all $u,v\in \frak{q}$.  Then the matrix representation of $\beta$ with respect to the basis $\{e_1,e_2,e_3,e_4\}$ is 
\begin{equation}
\nonumber
(\beta_{ij})=\left(\begin{array}{cccc}
0 & 0 & 0 & 2\\
0 & 0 & 1 & 0\\
0 & -1 & 0 & 0\\
-2 & 0 & 0 & 0\\
\end{array}\right)
\end{equation}
Hence, $\beta$ is nondegenerate which shows that $(\frak{q},\alpha)$ is a Frobenius Lie algbera.  Let $G$ be the set of linear isomorphisms on $\frak{q}$ whose matrix representations with respect to $\{e_1,e_2,e_3,e_4\}$ is given by
\begin{equation}
\label{gqFLA13e1}
\left\{\left(\begin{array}{cccc}
1 & 0 & 0 & 0\\
0 & b & c & 0\\
0 & 0 & 1/b & 0\\
a & 0 & 0 & 1\\
\end{array}\right)~\Big|~a,c\in \mathbb{R},~b>0\right\}.
\end{equation}
A direct calculation shows that $G$ is a 3-dimensional non-abelian, connected Lie subgroup of $Aut(\frak{q},\beta)$.  Let $\rho: G\rightarrow Aut(\frak{q},\beta)\subset GL(\frak{q})$ be the inclusion map (which is clearly a Lie group homomorphism).  Proposition \ref{gqFLA12} implies that $(\frak{q},\beta,\rho')$ is a $\frak{g}$-quasi-Frobenius Lie algebra, where $\rho':=\rho_{\ast,e}$.  As a Lie algebra, $\frak{g}$ has basis
\begin{equation}
\label{gqFLA13e2}
x_1:=\left(\begin{array}{cccc}
0 & 0 & 0 & 0\\
0 & 0 & 0 & 0\\
0 & 0 & 0 & 0\\
1 & 0 & 0 & 0\\
\end{array}\right),\hspace*{0.1in} x_2:=\left(\begin{array}{cccc}
0 & 0 & 0 & 0\\
0 & 1 & 0 & 0\\
0 & 0 & -1 & 0\\
0 & 0 & 0 & 0\\
\end{array}\right),\hspace*{0.1in} x_3:=\left(\begin{array}{cccc}
0 & 0 & 0 & 0\\
0 & 0 & 1 & 0\\
0 & 0 & 0 & 0\\
0 & 0 & 0 & 0\\
\end{array}\right),
\end{equation}
where we have identified $G$ with its matrix represenations in (\ref{gqFLA13e1}).  The non-zero commutator relations are
\begin{equation}
\nonumber
[x_2,x_3]=2x_3.
\end{equation}
Let $a=a_1x_1+a_2x_2+a_3x_3\in \frak{g}$.   Since $\rho:G\rightarrow Aut(\frak{q},\beta)\subset GL(\frak{q})$ is just the inclusion map, it follows that the matrix representation of $\rho'_a:\frak{q}\rightarrow \frak{q}$ with respect to the basis $\{e_1,e_2,e_3,e_4\}$ is simply 
\begin{equation}
\label{gqFLA13e3}
\rho'_{a}=\left(\begin{array}{cccc}
0 & 0 & 0 & 0\\
0 & a_2 & a_3 & 0\\
0 & 0 & -a_2 & 0\\
a_1 & 0 & 0 & 0\\
\end{array}\right)
\end{equation}
Since $(\frak{q},\beta,\rho')$ is a $\frak{g}$-quasi-Frobenius Lie algebra by Proposition \ref{gqFLA12}, $\rho'_a$ acts on $\frak{q}$ via derivations and satisfies 
\begin{equation}
\nonumber
\beta(\rho'_a(u),v)+\beta(u,\rho'_a(v))=0
\end{equation}
for all $u,v\in \frak{q}$.  
\end{example}

\noindent For later use, we conclude this section with the following natural definition:
\begin{definition}
\label{gqFLA14}
Let $(\frak{q},\beta,\phi)$ and $(\frak{r},\sigma,\mu)$ be $\frak{g}$-quasi-Frobenius Lie algebras.   A homomorphism  from $(\frak{q},\beta,\phi)$ to $(\frak{r},\sigma,\mu)$ is a homomorphism
\begin{equation}
\nonumber
\psi: (\frak{q},\beta)\rightarrow (\frak{r},\sigma)
\end{equation}
of quasi-Frobenius Lie algebras which is also $\frak{g}$-equivariant, that is,
\begin{equation}
\nonumber
\psi\circ\phi_x=\mu_x\circ \psi
\end{equation}
for all $x\in \frak{g}$.
\end{definition}

\subsection{Categorical Formulation}

\noindent  In this section, we apply the idea of categorification to quasi-Frobenius Lie algebras.   The upshot of this is the notion of a \textit{quasi-Frobenius Lie object}, which can be viewed as the analogue of a Frobenius object in the current setting.  The starting point for this particular step is the categorification of Lie algebra due to  Goyvaerts and Vercruysse \cite{GV}: 
\begin{definition}
\label{gqFLA1}
A \textit{Lie object} in an additive symmetric monoidal category $(\mathcal{C},\otimes, I,\Phi,l,r,c)$ is a pair $(L,b)$ where $L$ is an object of $\mathcal{C}$ and $b: L\otimes L\rightarrow L$ is a morphism such that 
\begin{itemize}
\item[(i)] $b+b\circ c=0_{L\otimes L,L}$
\item[(ii)] $b\circ (id_L\otimes b)\circ (id_{L\otimes (L\otimes L)}+c_{L\otimes L,L}\circ \Phi^{-1}_{L,L,L}+\Phi_{L,L,L}\circ c_{L,L\otimes L})=0_{L\otimes (L\otimes L),L}$
\end{itemize}
\end{definition}
\begin{remark}
With regard to the notation in Definition \ref{gqFLA1}, $\otimes$ is the monoidal product; $I$ is the unit object; $\Phi$ is the associator; $l$ and $r$ are the left and right unit maps respectively; and $c$ is the braiding.
\end{remark}
\begin{example}
\label{gqFLA2}
Let $\textbf{Vect}_{k}$ be the symmetric monoidal additive category of  finite dimensional vector spaces over $k$.  It follows readily from Definition \ref{gqFLA1} that a Lie object $(L,b)$ in $\textbf{Vect}_{k}$ is precisely a finite dimensional Lie algebra $L$ over $k$ with Lie bracket $[x,y]:=b(x,y)$.
\end{example}

\begin{definition}
\label{gqFLA3}
A \textit{quasi-Frobenius Lie object} in an additive symmetric monoidal category $(\mathcal{C},\otimes, I,\Phi,l,r,c)$ is a triple $(L,b,\overline{\beta})$ such that
\begin{itemize}
\item[(1)] $(L,b)$ is a Lie object.
\item[(2)] $L$ has a left dual object $L^\ast$ (where $\varepsilon: L^\ast \otimes L\rightarrow I$ and $\eta: I\rightarrow L\otimes L^\ast$ denote the evaluation and coevaluation morphisms respectively).
\item[(3)] $\overline{\beta}: L\stackrel{\sim}{\rightarrow} L^\ast$ is an isomorphism such that the induced morphism
\begin{equation}
\nonumber
\beta:=\varepsilon\circ (\overline{\beta}\otimes id_L): L\otimes L\rightarrow I,
\end{equation}
satisfies 
\begin{align}
\nonumber
\beta+\beta\circ c_{L,L}=0_{L\otimes L,I}
\end{align}
and
\begin{align}
\nonumber
\beta\circ (b\otimes id_L)\circ [id_{(L\otimes L)\otimes L}+\Phi^{-1}_{L,L,L}\circ c_{L\otimes L,L}+c_{L,L\otimes L}\circ \Phi_{L,L,L}]=0_{(L\otimes L)\otimes L,I}.
\end{align}
\end{itemize}
If there exists a morphism $\alpha: L\rightarrow I$ such that $\beta=\alpha \circ b$, then $(L,b,\overline{\beta})$ is called a \textit{Frobenius Lie object}.  
\end{definition}

\begin{example}
\label{gqFLA4}
Let $(L,b,\overline{\beta})$ be a quasi-Frobenius Lie object  in $\mathbf{Vect}_k$.  Then its easy to see that $L$ is a  quasi-Frobenius Lie algebra over $k$ with Lie bracket $[x,y]:=b(x,y)$ and $\beta: L\otimes L\rightarrow k$ (as defined in (3) of Definition \ref{gqFLA3}) is the nondegenerate 2-cocycle in the Lie algebra cohomology of $L$.  Likewise, a Frobenius Lie object in $\mathbf{Vect}_k$ is just a Frobenius Lie algebra.
\end{example}

\begin{proposition}
\label{gqFLA5}
The category $\mathbf{Rep}(\frak{g})$ of finite dimensional left $\frak{g}$-modules over $k$ is an additive symmetric monoidal category where every object has a left dual and
\begin{itemize}
\item[(i)] the monoidal product is the usual tensor product of left $\frak{g}$-modules and $\frak{g}$-linear maps
\item[(ii)] the unit object is $k$ with the trivial $\frak{g}$-action
\item[(iii)] the associator $\Phi$ is the trivial one
\item[(iv)] for any object $(V,\rho)$ in  $\mathbf{Rep}(\frak{g})$, the left and right morphisms $l_V: k\otimes V\stackrel{\sim}{\rightarrow} V$ and $r_V: V\otimes k\stackrel{\sim}{\rightarrow} V$ are the trivial ones
\item[(v)] for objects $(V,\rho)$, $(W,\phi)$ in $\mathbf{Rep}(\frak{g})$, the braiding $c_{V,W}: V\otimes W\stackrel{\sim}{\rightarrow} W\otimes V$ is simply the linear map that sends $v\otimes w\in V\otimes W$ to $w\otimes v\in W\otimes V$
\item[(vi)] the left dual of an object $(V,\rho)$ in $\mathbf{Rep}(\frak{g})$ is the dual representation $(V^\ast,\rho^\ast)$ (i.e.,  $\rho^\ast_x:=-\rho_x^t: V^\ast \rightarrow V^\ast$ for $x\in \frak{g}$, where $\rho^t_x$ is the dual or transpose of $\rho_x: V\rightarrow V$)
\item[(vii)] the evaluation morphism is $\varepsilon: V^\ast \otimes V\rightarrow k$, $\varepsilon(\xi,v):=\xi(v)$ and the coevaluation morphism is $\eta: k\rightarrow V\otimes V^\ast$, $1\mapsto \sum_i e_i\otimes\delta^i$ where $e_i$ is any basis of $V$ and $\delta^i$ is the corresponding dual basis.
\end{itemize}
\end{proposition}
\begin{proof}
It is an easy exercise to verify that $(\mathbf{Rep}(\frak{g}),\otimes,k,\Phi,l,r,c)$ satisfies all the axioms of an additive symmetric monoidal category.  
\end{proof}

\noindent The next result establishes the categorical formulation of $\frak{g}$-quasi-Frobenius Lie algberas.
\begin{proposition}
\label{gqFLA7}
A quasi-Frobenius Lie object in $\mathbf{Rep}(\frak{g})$ is a $\frak{g}$-quasi-Frobenius Lie algebra.
\end{proposition}

\begin{proof}
By definition, a quasi-Frobenius Lie object in $\mathbf{Rep}(\frak{g})$ consists of a representation $(\frak{q},\rho)$ of $\frak{g}$ together with $\frak{g}$-linear maps 
\begin{equation}
\nonumber
b: \frak{q}\otimes \frak{q}\rightarrow \frak{q},\hspace*{0.2in}\overline{\beta}: \frak{q}\stackrel{\sim}{\rightarrow} \frak{q}^\ast,
\end{equation}
which satisfy conditions (1) and (3) of Definition \ref{gqFLA3}. 

We begin by verifying that $(\frak{q},\beta)$ is a quasi-Frobenius Lie algebra.  To start, note that condition (1) of Definition \ref{gqFLA3} implies that $\frak{q}$ is a Lie algebra with Lie bracket $[u,v]:=b(u, v)$.  From  Definition \ref{gqFLA3}, the morphism $\beta: \frak{q}\otimes \frak{q}\rightarrow k$ is given explicitly as
\begin{equation}
\nonumber
\beta(u,v)=\varepsilon(\overline{\beta}(u),v)=\overline{\beta}(u)(v).
\end{equation}
Condition (3) of  Definition \ref{gqFLA3} implies that $\beta$ is a 2-cocycle of $\frak{q}$ with values in $k$ (where $\frak{q}$ acts trivially on $k$).  Furthermore, since $\overline{\beta}: \frak{q}\stackrel{\sim}{\rightarrow} \frak{q}^\ast$ is an isomorphism, it follows that $\beta$ is nondegenerate.  Hence, $(\frak{q},\beta)$ is a quasi-Frobenius Lie algebra.

Since $\overline{\beta}$ is $\frak{g}$-linear (being a morphism of $\mathbf{Rep}(\frak{g})$), we have
\begin{align}
\label{gqFA4e1}
\overline{\beta}(\rho_x (u))(v)=\rho^\ast_x(\overline{\beta}(u))(v)=-\overline{\beta}(u)(\rho_x (v)),~\hspace*{0.1in}\forall~u,v\in \frak{q}
\end{align}
where we recall that $\rho^\ast_x:=-\rho^t_x$.  Expressing the left and right most sides of (\ref{gqFA4e1}) in terms of $\beta$ gives 
\begin{equation}
\nonumber
\beta(\rho_x(u),v)=-\beta(u,\rho_x (v)),
\end{equation}
which proves the $\frak{g}$-invariance of $\beta$, that is, $\beta(\rho_x(u),v)+\beta(u,\rho_x (v))=0$. 

Since $b$ is also $\frak{g}$-linear, we also have
\begin{align}
\nonumber
\rho_x([u,v])&=\rho_x(b(u\otimes v))\\
\nonumber
&=b(\overline{\rho}_x(u\otimes v))\\
\nonumber
&=b(\rho_x(u)\otimes v)+b(u\otimes \rho_x(v))\\
\nonumber
&=[\rho_x(u),v]+[u,\rho_x(v)],
\end{align}
where $\overline{\rho}_x$  in the second equality denotes the induced left $\frak{g}$-module structure on $\frak{q}\otimes \frak{q}$.  Hence, $(\frak{q},\beta,\rho)$ is a $\frak{g}$-quasi-Frobenius Lie algebra.
\end{proof}

\section{The Geometry of $\frak{g}$-quasi-Frobenius Lie algebras}

\subsection{$G$-Symplectic Lie groups}

\begin{definition}
\label{GSLG1}
Let $G$ be a Lie group.  A \textit{$G$-symplectic Lie group} is a triple $(Q,\omega,\varphi)$ where  $(Q,\omega)$ is a symplectic Lie group and
\begin{equation}
\nonumber
\varphi: G\times Q\rightarrow Q,\hspace*{0.1in} (g,q)\mapsto \varphi_g(q):=\varphi(g,q)
\end{equation}
is a smooth left action on $Q$ such that $\varphi_g: (Q,\omega)\rightarrow (Q,\omega)$ is an isomorphism of symplectic Lie groups.  
\end{definition}
\begin{notation}
When dealing with multiple Lie groups, we will denote the identity element of each group simply as $e$ as opposed to $e_G$ for $G$, $e_Q$ for $Q$, and so on when there is no risk of confusion.   
\end{notation}

\begin{proposition}
\label{GSLG2}
Let $(Q,\omega,\varphi)$ be a $G$-symplectic Lie group with action 
\begin{equation}
\nonumber
\varphi: G\times Q\rightarrow Q,\hspace*{0.1in} (g,q)\mapsto \varphi_g(q):=\varphi(g,q).
\end{equation}
Define
\begin{equation}
\nonumber
\varphi': G\rightarrow GL(\frak{q}),\hspace*{0.1in}g\mapsto \varphi'_g:=(\varphi_g)_{\ast,e}:\frak{q}\rightarrow \frak{q}
\end{equation}
\begin{equation}
\nonumber
\varphi'': \frak{g}\rightarrow \frak{gl}(\frak{q}),\hspace*{0.1in}x\mapsto \varphi''_x:=(\varphi')_{\ast,e}(x):\frak{q}\rightarrow \frak{q}.
\end{equation}
Then
\begin{itemize}
\item[(i)] $\varphi'$ is a representation of $G$ on $\frak{q}$ such that $\varphi'_g\in  \mbox{Aut}(\frak{q},\omega_e)$ for all $g\in G$.
\item[(ii)] $(\frak{q},\omega_e,\varphi'')$ is a $\frak{g}$-quasi-Frobenius Lie algebra.  
\end{itemize}
\end{proposition}

\begin{proof}
Since $\varphi$ is a left action of $G$ on $Q$ and $\varphi_g(e)=e$ for all $g\in G$, we have  
\begin{align}
\nonumber
\varphi'_g\circ \varphi'_h&=(\varphi_g)_{\ast,e}\circ (\varphi_h)_{\ast,e}\\
\nonumber
&=(\varphi_g\circ \varphi_h)_{\ast,e}\\
\nonumber
&=(\varphi_{gh})_{\ast,e}\\
\nonumber
&=\varphi'_{gh}.
\end{align}
Hence, $\varphi'$ is a representation of $G$ on $\frak{q}$.  Furthermore, since $\varphi_g: Q\rightarrow Q$ is both a Lie group isomorphism and a symplectomorphism, it follows that $\varphi'_g:\frak{q}\rightarrow \frak{q}$ is a Lie algebra isomorphism and  
\begin{equation}
\nonumber
\omega_e(u,v)=((\varphi_g)^\ast\omega)_e(u,v)=\omega_e((\varphi_g)_{\ast,e}(u),(\varphi_g)_{\ast,e}(v))=\omega_e(\varphi'_g(u),\varphi'_g(v)),
\end{equation}
which shows that $\varphi'_g\in \mbox{Aut}(\frak{q},\omega_e)$ for all $g\in  G$.   This proves (i).

Statement (ii) follows from an application of Proposition \ref{gqFLA12} to the quasi-Frobenius Lie algebra  $(\frak{q},\omega_e)$ with Lie group homomorphism $\varphi': G\rightarrow \mbox{Aut}(\frak{q},\omega_e)\subset GL(\frak{q})$.  This completes the proof.
 \end{proof}

\begin{remark}
\label{GSLG3}
We will refer to $(\frak{q},\omega_e,\varphi'')$ in Proposition \ref{GSLG2} as the $\frak{g}$-quasi-Frobenius Lie algebra associated to the $G$-symplectic Lie group $(Q,\omega,\varphi)$.
\end{remark}

\noindent The next result provides a means of constructing $G$-symplectic Lie groups.  
\begin{proposition}
\label{GSLG7}
Let $(Q,\omega)$ be a simply connected symplectic Lie group, let $G$ be a Lie group, and let $\rho: G\rightarrow \mbox{Aut}(\frak{q},\omega_e)$, $g\mapsto \rho_g$ be a Lie group homomorphism.  Then there exists a unique smooth left-$G$ action 
\begin{equation}
\nonumber
\widehat{\rho}: G\times Q\rightarrow Q,\hspace*{0.1in} (g,q)\mapsto \widehat{\rho}_g(q),
\end{equation}
such that $(Q,\omega,\widehat{\rho})$ is a $G$-symplectic Lie group and $(\widehat{\rho}_g)_{\ast,e}=\rho_g$. In particular, if $G$ is any Lie subgroup of $\mbox{Aut}(\frak{q},\omega_e)$ and $G\neq \{e\}$, then $(Q,\omega)$ admits the structure of a $G$-symplectic  Lie group with a nontrivial $G$-action.
\end{proposition}
\begin{proof}
Let $\rho: G\rightarrow \mbox{Aut}(\frak{q},\omega_e)$, $g\mapsto \rho_g$ be a Lie group homomorphism.  Since $Q$ is simply connected and $\rho_g\in \mbox{Aut}(\frak{q},\omega_e)$ for all $g\in G$, it follows from Proposition \ref{SLG6a} that there exists a unique homomorphism of symplectic Lie groups
\begin{equation}
\nonumber
\widehat{\rho}_g: (Q,\omega)\rightarrow (Q,\omega)
\end{equation}
such that $(\widehat{\rho}_g)_{\ast,e}=\rho_g$ for all $g\in G$.  Furthermore, for $g,h\in G$, we have
\begin{align}
\nonumber
(\widehat{\rho}_g\circ \widehat{\rho}_h)_{\ast,e}&=(\widehat{\rho}_g)_{\ast,e}\circ (\widehat{\rho}_h)_{\ast,e}\\
\nonumber
&=\rho_g\circ \rho_h\\
\nonumber
&=\rho_{gh}\\
\label{SLGDrin10e0}
&=(\widehat{\rho}_{gh})_{\ast,e}.
\end{align}
Since $\widehat{\rho}_g\circ \widehat{\rho}_h$ and $\widehat{\rho}_{gh}$ are Lie group homomorphisms and $Q$ is connected, equation (\ref{SLGDrin10e0}) implies that 
\begin{equation}
\label{SLGDrin10e01}
\widehat{\rho}_g\circ \widehat{\rho}_h=\widehat{\rho}_{gh}.
\end{equation}
Hence, 
\begin{equation}
\nonumber
\widehat{\rho}: G\times Q\rightarrow Q,\hspace*{0.1in} (g,q)\mapsto \widehat{\rho}_g(q)
\end{equation}
is a left (not necessarily smooth) $G$-action.  We now show that $\widehat{\rho}$ is smooth.  To do this, set $\widehat{\rho}(g,q)=\widehat{\rho}_g(q)$ for $g\in G$, $q\in Q$ and let $U$ be an open neighborhood of $0\in \frak{q}$ such that 
\begin{equation}
\nonumber
\mbox{exp}|_{U}: U\stackrel{\sim}{\rightarrow} \mbox{exp}(U)
\end{equation}
is a diffeomorphism.  The naturality of the exponential map implies that 
\begin{equation}
\label{SLGDrin10e1}
\widehat{\rho}(g,q)=\mbox{exp}\circ \rho_g\circ (\mbox{exp}\big|_U)^{-1}(q),\hspace*{0.1in}\forall~(g,q)\in G\times \mbox{exp}(U).
\end{equation}
Since the right side of (\ref{SLGDrin10e1}) is smooth on $G\times \mbox{exp}(U)$, it follows that $\widehat{\rho}|_{G\times \mbox{exp}(U)}$ is also smooth.  Now fix an arbitrary element $q_0$ of $Q$ and define
\begin{equation}
\nonumber
f: G\rightarrow Q,~\hspace*{0.1in} g\mapsto \widehat{\rho}(g,q_0).
\end{equation}
We now show that $f$ is smooth. Since $Q$ is connected, $\mbox{exp}(U)$ generates $Q$.  Hence, there exists $q_{0,1},\dots, q_{0,k}\in \mbox{exp}(U)$ such that 
\begin{equation}
\nonumber
q_0=q_{0,1}q_{0,2}\cdots q_{0,k}.
\end{equation}
Since $\widehat{\rho}_g: Q\rightarrow Q$ is a Lie group homomorphism for all $g\in G$, we have
\begin{equation}
\label{SLGDrin10e2}
f(g):=\widehat{\rho}(g,q_0)=\widehat{\rho}(g,q_{0,1})\widehat{\rho}(g,q_{0,2})\cdots \widehat{\rho}(g,q_{0,k})\in Q.
\end{equation}
Since $(g,q_{0,i})\in G\times  \mbox{exp}(U)$ for $i=1,\dots, k$, it follows that the right side of (\ref{SLGDrin10e2}) depends smoothly on $g$.  Hence, $f$ is smooth.   Now, for all $(g,q)\in G\times (q_0\mbox{exp}(U))$, we have
\begin{align}
\nonumber
\widehat{\rho}(g,q)&=\widehat{\rho}(g,q_0q_0^{-1}q)\\
\nonumber
&=\widehat{\rho}(g,q_0)\widehat{\rho}(g,q_0^{-1}q)\\
\label{SLGDrin10e3}
&=f(g)[(\widehat{\rho}|_{G\times \mbox{exp}(U)})\circ (id_G\times l_{q_0^{-1}})(g,q)],
\end{align}
where $l_{q_0^{-1}}: Q\rightarrow Q$ is left translation by $q_0^{-1}$.  Since $f$ and $\widehat{\rho}|_{G\times \mbox{exp}(U)}$ are both smooth, it follows that the right side of (\ref{SLGDrin10e3}) is smooth on $G\times (q_0\mbox{exp}(U))$.  Hence, $\widehat{\rho}|_{G\times (q_0\mbox{exp}(U))}$ is smooth.  Since $q_0\in Q$ is arbitrary, it follows that $\widehat{\rho}$ is smooth on $G\times Q$.  This completes the proof.
\end{proof}

\noindent We now illustrate Proposition \ref{GSLG7} with a simple example:

\begin{example}
\label{GSLG5}
Let $Q$ be the 2-dimensional non-abelian Lie group
\begin{equation}
\label{SLGDrin5e1}
Q=\left\{
\left(\begin{array}{cc}
a & b\\
0 & 1
\end{array}
\right)~\big| ~a>0, ~b\in \mathbb{R}
\right\}.
\end{equation}
Note that $Q$ is simply connected, being diffeomorphic to $\mathbb{R}_+\times \mathbb{R}$.  The associated Lie algebra is 
\begin{equation}
\label{SLGDrin5e2}
\frak{q}=\left\{
\left(\begin{array}{cc}
\overline{a} & \overline{b}\\
0 & 0
\end{array}
\right)~\big| ~\overline{a},~\overline{b}\in \mathbb{R}
\right\}.
\end{equation}
A convenient basis for $\frak{q}$ is then 
\begin{equation}
\label{SLGDrin5e3}
e_1=
\left(\begin{array}{cc}
1 & 0\\
0 & 0
\end{array}
\right),\hspace*{0.2in} e_2=\left(\begin{array}{cc}
0 & 1\\
0 & 0
\end{array}
\right),
\end{equation}
where we note that 
\begin{equation}
\label{SLGDrin5e4}
[e_1,e_2]=e_2.
\end{equation}
Let $\alpha: \frak{q}\rightarrow \mathbb{R}$ be the linear map defined by $\alpha(e_1)=0$ and $\alpha(e_2)=1$.  Then $(\frak{q},\alpha)$ is a Frobenius Lie algebra.  Let $\widetilde{\beta}$ be the left-invariant symplectic form on $Q$ defined by $\widetilde{\beta}_e=\beta$, where $\beta(u,v):=\alpha([u,v])$ for $u,v\in \frak{q}$.  

For $\lambda\in \mathbb{R}$, let $\rho_\lambda: \frak{q}\rightarrow \frak{q}$ be the linear isomorphism defined by
\begin{equation}
\nonumber
\rho_\lambda(e_1):=e_1+\lambda e_2,\hspace*{0.1in} \rho_\lambda(e_2):=e_2.
\end{equation}
Then it is a straightforward exercise to show that $\rho_\lambda\in Aut(\frak{q},\omega_e)$ and
\begin{equation}
\nonumber
\rho: \mathbb{R}\stackrel{\sim}{\rightarrow} Aut(\frak{q},\omega_e),\hspace*{0.1in} \lambda \mapsto\rho_\lambda
\end{equation}
is a Lie group isomorphism.  Proposition \ref{GSLG7} implies that $(Q,\omega)$ admits the structure of an $\mathbb{R}$-symplectic Lie group with unique action $\widehat{\rho}: \mathbb{R}\times Q\rightarrow Q$ satisfying $(\widehat{\rho}_\lambda)_{\ast,e}=\rho_\lambda$.  

We now compute the action $\widehat{\rho}$ explicitly. Let $u\in \frak{q}$.  Then 
\begin{equation}
\label{SLGDrin8e2}
u=\overline{a} e_1+\overline{b} e_2=\left(\begin{array}{cc}
\overline{a} & \overline{b}\\
0 & 0
\end{array}
\right)
\end{equation}
for some $a,b\in \mathbb{R}$.  Using the naturality of the exponential map, we have
\begin{equation}
\label{SLGDrin8e3}
\widehat{\rho}_\lambda\circ \mbox{exp}(u)=\mbox{exp}\circ \rho_\lambda(u).
\end{equation}
A direct calculation shows that 
\begin{equation}
\label{SLGDrin8e4}
\mbox{exp}(u)=\left(\begin{array}{cc}
e^{\overline{a}} & \mu(\overline{a})\overline{b}\\
0 & 1
\end{array}\right),
\end{equation}
where $\mu: \mathbb{R}\rightarrow \mathbb{R}_+$ is the nonzero smooth function given by $\mu(t)=\frac{1}{t}(e^t-1)$ for $t\neq 0$ and $\mu(0)=1$.   Note that every element of $Q$ is in the image of the exponential map.  Indeed, given 
\begin{equation}
\nonumber
q=\left(\begin{array}{cc}
a & b\\
0 & 1
\end{array}\right)
\end{equation}
for $a>0$, $b\in \mathbb{R}$, one simply sets $\overline{a}=\ln a$ and $\overline{b}=b/\mu(\ln a)$  in (\ref{SLGDrin8e4}) to obtain $\mbox{exp}(u)=q$.  The left side of (\ref{SLGDrin8e3}) is 
\begin{equation}
\label{SLGDrin8e5}
\mbox{exp}\circ \rho_\lambda(u)=\mbox{exp}\left(\begin{array}{cc}
\overline{a} & \lambda \overline{a} + \overline{b}\\
0 & 0
\end{array}\right)=\left(\begin{array}{cc}
e^{\overline{a}} & \mu(\overline{a})(\lambda \overline{a} + \overline{b})\\
0 & 1
\end{array}\right).
\end{equation}
Hence,
\begin{equation}
\label{SLGDrin8e6}
\widehat{\rho}_\lambda \left(\begin{array}{cc}
e^{\overline{a}} & \mu(\overline{a})\overline{b}\\
0 & 1
\end{array}\right)=\left(\begin{array}{cc}
e^{\overline{a}} & \mu(\overline{a})(\lambda \overline{a} + \overline{b})\\
0 & 1
\end{array}\right).
\end{equation}
Setting $\overline{a}=\ln a$ and $\overline{b}=b/\mu(\ln a)$ for $a>0$ and $b\in \mathbb{R}$, we obtain
\begin{equation}
\label{SLGDrin8e7}
\widehat{\rho}_\lambda \left(\begin{array}{cc}
a & b\\
0 & 1
\end{array}\right)=\left(\begin{array}{cc}
a & \lambda(a-1)+b\\
0 & 1
\end{array}\right).
\end{equation}
Since $(Q,\omega,\varphi)$ is an $\mathbb{R}$-symplectic Lie group by Proposition \ref{GSLG7}, $\widehat{\rho}_\lambda$ is both a Lie group isomorphism and a symplectomorphism of $(Q,\omega)$ which satisfies $(\widehat{\rho}_\lambda)_{\ast,e}=\rho_\lambda$. 
\end{example}

\noindent In anticipation of the next section, we introduce the following definition:

\begin{definition}
\label{GSLG8}
Let $(Q,\omega,\varphi)$ and $(R,\tau,\chi)$ be $G$-symplectic Lie groups.  A homomorphism of $G$-symplectic Lie groups from $(Q,\omega,\varphi)$ to $(R,\tau,\chi)$ is a homomorphism 
\begin{equation}
\nonumber
\Psi: (Q,\omega)\rightarrow (R,\tau)
\end{equation}
of symplectic Lie groups which is also $G$-equivariant, that is, $\Psi(\varphi_g(q))=\chi_g(\Psi(q))$ for all $g\in G$ and $q\in Q$.
\end{definition}

\subsection{The Equivalence}
\noindent In this section, we show that the category of finite dimensional $\frak{g}$-quasi-Frobenius Lie algebras is equivalent to the category of simply connected $G$-symplectic Lie groups, where $G$ is also simply connected.   We begin with the following result.

\begin{proposition}
\label{GSLG9}
Let $\Psi: (Q,\omega,\varphi)\rightarrow (R,\tau,\chi)$ be a homomorphism of $G$-symplectic Lie groups.  Then 
\begin{equation}
\nonumber
\Psi_{\ast,e}: (\frak{q},\omega_e,\varphi'')\rightarrow (\frak{r},\tau_e,\chi'')
\end{equation}
is a homomorphism of $\frak{g}$-quasi-Frobenius Lie algebras, where $\varphi''$ and $\chi''$ are defined as in Proposition \ref{GSLG2}.
\end{proposition}

\begin{proof}
By definition, $\Psi: (Q,\omega)\rightarrow (R,\tau)$ is a homomorphism of symplectic Lie groups.  This implies that 
\begin{equation}
\nonumber
\Psi_{\ast,e}: (\frak{q},\omega_e)\rightarrow (\frak{r},\tau_e)
\end{equation}
is a homomorphism of quasi-Frobenius Lie algebras.  It only remains to show that $\Psi_{\ast,e}$ is $\frak{g}$-equivariant.   Since $\Psi$ is $G$-equivariant, we have
\begin{equation}
\nonumber
\Psi\circ \varphi_{g}=\chi_{g}\circ \Psi,\hspace*{0.1in} \forall~g\in G.
\end{equation}
This in turn implies that 
\begin{equation}
\label{GSLG9e1}
\Psi_{\ast,e}\circ \varphi'_g=\chi'_g\circ \Psi_{\ast,e},\hspace*{0.1in} \forall~g\in G,
\end{equation}
where $\varphi'_g:=(\varphi_g)_{\ast,e}:\frak{q}\rightarrow \frak{q}$ and $\chi'_g:=(\chi_g)_{\ast,e}:\frak{r}\rightarrow \frak{r}$.  Let $x\in \frak{g}$ and set $g=\mbox{exp}(tx)$ in (\ref{GSLG9e1}).  Applying $\frac{d}{dt}\big|_{t=0}$ to both sides then gives
\begin{equation}
\label{GSLG9e1}
\Psi_{\ast,e}\circ \varphi''_x=\chi''_x\circ \Psi_{\ast,e}.
\end{equation}
This in turn completes the proof.
\end{proof}

\begin{lemma}
\label{GSLG10}
Let $(\frak{q},\beta,\phi)$ be a $\frak{g}$-quasi-Frobenius Lie algebra and let $G$ be the simply connected Lie group whose Lie algebra is $\frak{g}$.  Then there exists a unique Lie group homomorphism $f: G\rightarrow GL(\frak{q})$, $g\mapsto f_g$ such that $f_{\ast,e}=\phi$ and $f_g\in Aut(\frak{q},\beta)$ for all $g\in G$.
\end{lemma}

\begin{proof}
Since $G$ is simply connected and $\phi: \frak{g}\rightarrow \frak{gl}(\frak{q})$ is a Lie algebra map, there exists a unique Lie group homomorphism $f: G\rightarrow GL(\frak{q})$ such that $f_{\ast,e}=\phi$.  We now show that $f_g\in Aut(\frak{q},\beta)$ for all $g\in G$.  Fix $x\in \frak{g}$.  To simplify notation, let 
\begin{equation}
\label{GSLG10e2}
f_t:=f_{\mbox{exp}(tx)}: \frak{q}\rightarrow \frak{q}.
\end{equation}
Define $A:  \mathbb{R}\times  \frak{q}\times \frak{q}\rightarrow \mathbb{R}$ by
\begin{equation}
\label{GSLG10e1}
A(t,u,v):=\beta(f_t(u),f_t(v))-\beta(u,v).
\end{equation}
Since $f_{\ast,e}=\phi$ and $(\frak{q},\beta,\phi)$ is a $\frak{g}$-quasi-Frobenius Lie algebra, we have
\begin{equation}
\label{GSLG10e2}
\frac{d}{dt}\big|_{t=0}A(t,u,v)=\beta(\phi_x(u),v)+\beta(u,\phi_x(v))=0,\hspace*{0.1in}\forall ~u,v\in \frak{q}.
\end{equation}
Furthermore, since $f$ is a group homomorphism and 
\begin{equation}
\nonumber
 \mbox{exp}((t+s)x)=\mbox{exp}(tx)\mbox{exp}(sx),
\end{equation}
we have
\begin{equation}
\label{GSLG10e3}
A(t+s,u,v)=A(t,f_s(u),f_s(v))+A(s,u,v),\hspace*{0.1in}\forall ~u,v\in \frak{q}.
\end{equation}
 Equations (\ref{GSLG10e2}) and (\ref{GSLG10e3}) imply
 \begin{equation}
 \label{GSLG10e4}
 \frac{d}{dt}\big|_{t=s}A(t,u,v)=\frac{d}{dt}\big|_{t=0}A(t+s,u,v)=0+0=0.
 \end{equation}
Hence, for fixed $u,v\in \frak{q}$, $A(t,u,v)$ is a constant.  Since $A(0,u,v)=0$, it follows that $A(t,u,v)=0$ for all $t\in \mathbb{R}$.  Hence,
\begin{equation}
 \label{GSLG10e5}
 \beta(f_t(u),f_t(v))=\beta(u,v),\hspace*{0.1in} \forall ~t\in \mathbb{R}.
\end{equation}
In particular,
\begin{equation}
 \label{GSLG10e5a}
 \beta(f_{\mbox{exp}(x)}(u),f_{\mbox{exp}(x)}(v))=\beta(u,v).
\end{equation}

Now define $B: \mathbb{R}\times \frak{q}\times \frak{q}\times \frak{q}\rightarrow \frak{q}$ by
\begin{equation}
 \label{GSLG10e6}
B(t,u,v,w)=\beta([f_t(u),f_t(v)]-f_t([u,v]),f_t(w)).
\end{equation}
Equation (\ref{GSLG10e5}) implies that
\begin{equation}
 \label{GSLG10e7}
B(t,u,v,w)=\beta([f_t(u),f_t(v)],f_t(w))-\beta([u,v],w).
\end{equation}
Using (\ref{GSLG10e7}) and the fact that $(\frak{q},\beta,\phi)$ is a $\frak{g}$-quasi-Frobenius Lie algebra, we have
\begin{align}
\nonumber
\frac{d}{dt}\big|_{t=0} B(t,u,v,w)&=\beta([\phi_x(u),v],w)+\beta([u,\phi_x(v)],w)+\beta([u,v],\phi_x(w))\\
\nonumber
&=\beta(\phi_x([u,v]),w)+\beta([u,v],\phi_x(w))\\
 \label{GSLG10e8}
&=0,\hspace*{0.1in}\forall~u,v,w.
\end{align}
From (\ref{GSLG10e7}), we also have
\begin{align}
 \label{GSLG10e9}
 &B(t+s,u,v,w)=\beta([f_t(f_s(u)),f_t(f_s(v))],f_t(f_s(w)))-\beta([u,v],w)\\
 \nonumber
 &=B(t,f_s(u),f_s(v),f_s(w))+\beta([f_s(u),f_s(v)],f_s(w))-\beta([u,v],w)
 \end{align}
 Equations (\ref{GSLG10e8}) and (\ref{GSLG10e9}) now imply
 \begin{equation}
 \label{GSLG10e10}
 \frac{d}{dt}\big|_{t=s}B(t,u,v,w)= \frac{d}{dt}\big|_{t=0}B(t+s,u,v,w)=0+0+0=0.
 \end{equation}
From (\ref{GSLG10e10}), it follows that for fixed $u,v,w$, $B(t,u,v,w)$ is a constant for all $t\in \mathbb{R}$.  Hence, $B(t,u,v,w)=B(0,u,v,w)=0$ for all $t\in \mathbb{R}$ and $u,v,w\in \frak{q}$. In particular,
 \begin{equation}
 \label{GSLG10e11}
B(1,u,v,w)=\beta([f_1(u),f_1(v)]-f_1([u,v]),f_1(w))=0,\hspace*{0.1in} \forall~u,v,w\in \frak{q}.
\end{equation}
 Since $\beta$ is non-degenerate and $f_1:=f_{\mbox{exp}(x)}\in GL(\frak{q})$, it follows that
 \begin{equation}
 \label{GSLG10e12}
 f_{\mbox{exp}(x)}([u,v])=[ f_{\mbox{exp}(x)}(u), f_{\mbox{exp}(x)}(v)].
 \end{equation}
 
 Since $G$ is connected, $x\in \frak{g}$ is arbitrary, and $f$ is a group homomorphism, equations (\ref{GSLG10e5a}) and (\ref{GSLG10e12}) imply that 
 \begin{align}
  \label{GSLG10e13}
  \beta(f_g(u),f_g(v))=\beta(u,v),\hspace*{0.1in} f_g([u,v])=[f_g(u),f_g(v)]
 \end{align}
 for all $g\in G$.  Hence, $f_g\in \mbox{Aut}(\frak{q},\beta)$ for all $g\in G$.  This completes the proof.
 \end{proof}

\begin{proposition}
\label{GSLG11}
Let $(\frak{q},\beta,\phi)$ be a $\frak{g}$-quasi-Frobenius Lie algebra.  Let $G$ and $Q$ be the simply connected Lie groups associated to $\frak{g}$ and $\frak{q}$ respectively and let $\widetilde{\beta}\in \Omega^2(Q)$ be the left-invariant 2-form associated to $\beta$.  Then there exists a unique left action $\overline{\phi}: G\times Q\rightarrow Q$ such that $(Q,\widetilde{\beta},\overline{\phi})$ is a $G$-symplectic Lie group whose associated $\frak{g}$-quasi-Frobenius Lie algebra is 
\begin{equation}
\nonumber
(\frak{q},\widetilde{\beta}_e,\overline{\phi}'')=(\frak{q},\beta,\phi),
\end{equation}
where $\overline{\phi}''$ is defined as in Proposition \ref{GSLG2}.
\end{proposition}
\begin{proof}
By Proposition \ref{SLG3}, $(Q,\widetilde{\beta})$ is a symplectic Lie group.  Since $G$ is simply connected, Lemma \ref{GSLG10} shows that there exists a unique Lie group homomorphism 
\begin{equation}
\nonumber
f: G \rightarrow GL(\frak{q}),\hspace*{0.1in} g\mapsto f_g
\end{equation}
such that $f_{\ast,e}=\phi: \frak{g}\rightarrow \frak{gl}(\frak{q})$ and $f_g\in Aut(\frak{q},\beta)$ for all $g\in G$.  Since $Q$ is simply connected, Proposition \ref{GSLG7} shows that there exists a unique smooth left $G$-action 
\begin{equation}
\nonumber
\overline{\phi}: G\times Q\rightarrow Q,\hspace*{0.2in} (g,q)\mapsto \overline{\phi}_g(q)
\end{equation}
such that $(Q,\widetilde{\beta},\overline{\phi})$ is a $G$-symplectic Lie group and $(\overline{\phi}_g)_{\ast,e}=f_g$.  Setting $\overline{\phi}'_g:=(\overline{\phi}_g)_{\ast,e}$ as in Proposition \ref{GSLG2}, we have 
\begin{equation}
\nonumber
\overline{\phi}'':=\overline{\phi}'_{\ast,e}=f_{\ast,e}=\phi.
\end{equation}
This completes the proof.
\end{proof}

\begin{proposition}
\label{GSLG12}
Let $\psi: (\frak{q},\beta,\phi)\rightarrow (\frak{r},\sigma,\mu)$ be a homomorphism of $\frak{g}$-quasi-Frobenius Lie algebras.  Let $G$ be the simply connected Lie group whose Lie algebra is $\frak{g}$ and let $(Q,\widetilde{\beta},\overline{\phi})$ and $(R,\widetilde{\sigma},\overline{\mu})$ be the simply connected $G$-symplectic Lie groups associated to $ (\frak{q},\beta,\phi)$ and $(\frak{r},\sigma,\mu)$ respectively by Proposition \ref{GSLG11}.  Then there exists a unique homomorphism of $G$-symplectic Lie groups
\begin{equation}
\nonumber
\widehat{\psi}: (Q,\widetilde{\beta},\overline{\phi})\rightarrow (R,\widetilde{\sigma},\overline{\mu})
\end{equation}
 such that $\widehat{\psi}_{\ast,e}=\psi$.
\end{proposition}

\begin{proof}
By Proposition \ref{SLG6a}, there exists a unique homomorphism of symplectic Lie groups $\widehat{\psi}: (Q,\widetilde{\beta})\rightarrow (R,\widetilde{\sigma})$ such that $\widehat{\psi}_{\ast,e}=\psi$.  We now verify that $\widehat{\psi}$ is $G$-equivariant.  

Let $\overline{\phi}': G\rightarrow Aut(\frak{q},\beta)$, $g\mapsto \overline{\phi}'_g$ and $\overline{\mu}': G\rightarrow Aut(\frak{r},\sigma)$, $g\mapsto \overline{\mu}'_g$ be defined as in Proposition \ref{GSLG2}.  Fix $x\in \frak{g}$.  To simplify notation, let 
\begin{equation}
\nonumber
\overline{\phi}'_t:=\overline{\phi}'_{\mbox{exp}(tx)},\hspace*{0.1in} \overline{\mu}'_t:=\overline{\mu}'_{\mbox{exp}(tx)}.
\end{equation}
Define $B: \mathbb{R}\times \frak{q}\times \frak{r}\rightarrow \mathbb{R}$ by
\begin{align}
\nonumber
B(t,u,v)&:=\sigma(\psi\circ \overline{\phi}'_t(u)-\overline{\mu}'_t\circ \psi(u),\overline{\mu}'_t(v))\\
\nonumber
&=\sigma(\psi\circ \overline{\phi}'_t(u),\overline{\mu}'_t(v))-\sigma(\overline{\mu}'_t\circ \psi(u),\overline{\mu}'_t(v))\\
\label{GSLG12e1}
&=\sigma(\psi\circ \overline{\phi}'_t(u),\overline{\mu}'_t(v))-\sigma(\psi(u),v),
\end{align}
where the third equality follows from the fact that $\overline{\mu}'_t\in  Aut(\frak{r},\sigma)$.  Hence,
\begin{align}
\nonumber
\frac{d}{dt}\big|_{t=0} B(t,u,v)&=\sigma(\psi\circ \phi_x(u),v)+\sigma(\psi(u),\mu_x(v))\\
\nonumber
&=\sigma(\mu_x\circ \psi(u),v)+\sigma(\psi(u),\mu_x(v))\\
\label{GSLG12e2}
&=0,~\hspace*{0.1in}\forall~u\in \frak{q},~v\in \frak{r}
\end{align}
where the second equality follows from the fact that $\psi$ is $\frak{g}$-equivariant (i.e., $\psi\circ \phi_x=\mu_x\circ \psi$) and the third equality follows from the fact that $(\frak{r},\sigma,\mu)$ is a $\frak{g}$-quasi-Frobenius Lie algebra with 2-cocycle $\sigma$ and $\frak{g}$-action $\mu$.  Next note that 
\begin{align}
\label{GSLG12e3}
B(t+s,u,v)&=B(t,\overline{\phi}'_s(u),\overline{\mu}'_s(v))+\sigma(\psi(\overline{\phi}'_s(u)),\overline{\mu}'_s(v))-\sigma(\psi(u),v)
\end{align}
Hence,
\begin{align}
\label{GSLG12e4}
\frac{d}{dt}\big|_{t=s}B(t,u,v)=\frac{d}{dt}\big|_{t=0}B(t+s,u,v)=0+0-0=0,
\end{align}
where the first zero follows from (\ref{GSLG12e2}).  Hence, 
\begin{equation}
\label{GSLG12e5}
B(t,u,v)=B(0,u,v)=0,\hspace*{0.1in}\forall t\in \mathbb{R},~u\in \frak{q}~,v\in \frak{r}.
\end{equation}
In particular, $B(1,u,v)=0$ for all $u,v\in \frak{r}$.  Since $\sigma$ is nondegenerate and $\overline{\mu}'_t: \frak{r}\rightarrow \frak{r}$ is also a linear isomorphism for all $t$, it follows that
\begin{equation}
\label{GSLG12e6}
\psi\circ \overline{\phi}'_t=\overline{\mu}'_t\circ \psi,\hspace*{0.1in}\forall t\in \mathbb{R}.
\end{equation}
In particular, we have
\begin{equation}
\label{GSLG12e7}
\psi\circ \overline{\phi}'_{\mbox{exp}(x)}=\overline{\mu}'_{\mbox{exp}(x)}\circ \psi.
\end{equation}
Since $x\in \frak{g}$ was arbitrary, (\ref{GSLG12e7}) must hold for all $x\in \frak{g}$.  Since $G$ is connected, every element $g\in G$ is of the form $g=\mbox{exp}(x_1)\cdots \mbox{exp}(x_k)$ for some $x_i\in \frak{g}$, $i=1,\dots, k$.   It follows from this and the fact that $\overline{\phi}'$ and $\overline{\mu}'$ are group homomorphisms that 
\begin{equation}
\label{GSLG12e8}
\psi\circ \overline{\phi}'_{g}=\overline{\mu}'_{g}\circ \psi,\hspace*{0.1in}\forall ~g\in G.
\end{equation}
Equation (\ref{GSLG12e8}) combined with the fact that  (1) $Q$ is connected, (2) $\widehat{\psi}\circ \overline{\phi}_g$ and $\overline{\mu}_g\circ \widehat{\psi}$ are both Lie group homomorphisms $\forall ~g\in G$, and (3)
\begin{equation}
\label{GSLG12e9}
(\widehat{\psi}\circ \overline{\phi}_g)_{\ast,e}=\psi\circ \overline{\phi}'_g=\overline{\mu}'_{g}\circ \psi=(\overline{\mu}_g\circ \widehat{\psi})_{\ast,e}, \hspace*{0.1in}\forall ~g\in G
\end{equation}
imply that $\widehat{\psi}\circ \overline{\phi}_g=\overline{\mu}_g\circ \widehat{\psi}$ for all $g\in G$.  In other words, $\widehat{\psi}$ is $G$-equivariant and this completes the proof.
\end{proof}

\noindent We conclude the paper with the following generalization of Theorem \ref{SLG7}.
\begin{theorem}
\label{GSLG13}
Let $G$ be a simply connected Lie group and let $G$-$\mathbf{SCSLG}$ be the category of simply connected $G$-symplectic Lie groups and let $\frak{g}$-$\mathbf{qFLA}$ be the category of finite dimensional $\frak{g}$-quasi-Frobenius Lie algebras.  Let $\widehat{F}$ be the functor from $G$-$\mathbf{SCSLG}$ to $\frak{g}$-$\mathbf{qFLA}$ which sends the object $(Q,\omega,\varphi)$ to $(\frak{q},\omega_e,\varphi'')$, where $\varphi''$ is defined as in Proposition \ref{GSLG2} and the morphism $\Psi: (Q,\omega,\varphi)\rightarrow (R,\tau,\chi)$ to 
\begin{equation}
\nonumber
\Psi_{\ast,e}: (\frak{q},\omega_e,\varphi'')\mapsto (\frak{r},\tau_e,\chi'').
\end{equation}
Then $\widehat{F}$ is an equivalence of categories.  
\end{theorem}

\begin{proof}
Theorem \ref{GSLG13} follows from Theorem \ref{SLG7}, Proposition \ref{GSLG2}, Proposition \ref{GSLG9}, Proposition \ref{GSLG11}, and Proposition \ref{GSLG12}.
\end{proof}

\section{$D(\frak{g})$-quasi-Frobenius Lie algberas}
\noindent Let $(\frak{g},\gamma)$ be a finite dimensional Lie bialgebra.  We begin with the following observation:

\begin{proposition}
\label{Dg1}
Let $V$ be a vector space over $k$ and let $\rho: D(\frak{g})\rightarrow \frak{gl}(V)$ be a linear map (not necessarily a representation).  Define
\begin{equation}
\nonumber
\varphi:=\rho|_{\frak{g}}: \frak{g} \rightarrow \frak{gl}(V),\hspace*{0.2in}\psi:=\rho|_{\frak{g}^\ast}: \frak{g}^\ast \rightarrow \frak{gl}(V).
\end{equation}
The following statements are equivalent.
\begin{itemize}
\item[(i)] $\rho$ is a representation of $D(\frak{g})$ on $V$.
\item[(ii)] $\varphi$ and $\psi$ are representations of $\frak{g}$ and $\frak{g}^\ast$ on $V$ which satisfy
\begin{equation}
\label{Dg1e1}
\psi_{ad^\ast_x\xi}-\varphi_{ad^\ast_\xi x}=\varphi_x\circ \psi_\xi-\psi_\xi\circ \varphi_x,~\hspace*{0.1in}\forall~x\in \frak{g},~\xi\in \frak{g}^\ast.
\end{equation}
\end{itemize}
\end{proposition}
\begin{proof}
$(i)\Rightarrow (ii)$.   Since $\rho$ is a representation of $D(\frak{g})$ on $V$, it follows immediately that $\varphi$ and $\psi$ must be representations of $\frak{g}$ and $\frak{g}^\ast$ on $V$ respectively.   For  (\ref{Dg1e1}), we note that 
\begin{equation}
\nonumber
[x,\xi]=ad^\ast_x\xi-ad^\ast_\xi x\hspace*{0.1in} \forall~ x\in \frak{g},~\xi\in \frak{g}^\ast. 
\end{equation}
Since $\rho$ is a representation and $\varphi:=\rho|_{\frak{g}}$ and $\psi:=\rho|_{\frak{g}^\ast}$, we have
\begin{equation}
\nonumber
\psi_{ad^\ast_x\xi}-\varphi_{ad^\ast_\xi x}=\rho_{[x,\xi]}=\rho_x\rho_\xi-\rho_\xi\rho_x=\varphi_x\psi_\xi-\psi_\xi\varphi_x,
\end{equation}
which proves (\ref{Dg1e1}).

$(i)\Leftarrow (ii)$. Let $a=x+\xi\in D(\frak{g})$.  Then
\begin{align}
\nonumber
\rho_{[x+\xi,y+\eta]}&=\rho_{[x,y]}+\rho_{[x,\eta]}+\rho_{[\xi,y]}+\rho_{[\xi,\eta]}\\
\nonumber
&=\varphi_{[x,y]}+\psi_{ad^\ast_x\eta}-\varphi_{ad^\ast_\eta x}+\varphi_{ad^\ast_{\xi}y}-\psi_{ad^\ast_y \xi}+\psi_{[\xi,\eta]} \\
\nonumber
&=\varphi_x\circ \varphi_y-\varphi_y\circ \varphi_x+\varphi_x\circ\psi_\eta-\psi_\eta\circ \varphi_x\\
\nonumber
&\hspace*{0.2in} + \psi_\xi\circ \varphi_y-\varphi_y\circ \psi_\xi+\psi_\xi\circ \psi_\eta-\psi_\eta\circ \psi_\xi\\
\nonumber
&=(\varphi_x+\psi_\xi)\circ (\varphi_y+\psi_\eta)- (\varphi_y+\psi_\eta)\circ (\varphi_x+\psi_\xi)\\
\nonumber
&=\rho_{x+\xi}\circ \rho_{y+\eta}- \rho_{y+\eta}\circ \rho_{x+\xi}.
\end{align}
This proves that $\rho: D(\frak{g})\rightarrow \frak{gl}(V)$ is a representation of $D(\frak{g})$ on $V$. 
\end{proof}

\begin{proposition}
\label{Dg2}
Let $(\frak{q},\beta)$ be a quasi-Frobenius Lie algebra and let $\rho: D(\frak{g})\rightarrow \frak{gl}(\frak{q})$ be a linear map (not necessarily a representation).  Define $\varphi:=\rho|_{\frak{g}}$ and $\psi:=\rho|_{\frak{g}^\ast}$.  Then $(\frak{q},\beta,\rho)$ is a $D(\frak{g})$-quasi-Frobenius Lie algebra iff the following conditions are satisfied:
\begin{itemize}
\item[(a)] $\psi_{ad^\ast_x\xi}-\varphi_{ad^\ast_\xi x}=\varphi_x\circ \psi_\xi-\psi_\xi\circ \varphi_x,~\hspace*{0.1in}\forall~x\in \frak{g},~\xi\in \frak{g}^\ast$
\item[(b)] $(\frak{q},\beta,\varphi)$ is a $\frak{g}$-quasi-Frobenius Lie algebra.
\item[(b)] $(\frak{q},\beta,\psi)$ is a $\frak{g}^\ast$-quasi-Frobenius Lie algebra.
\end{itemize}
\end{proposition}

\begin{proof}
By Proposition \ref{Dg1}, $\rho$ is left $D(\frak{g})$-module structure on $\frak{q}$ iff $\varphi$ and $\psi$ are left $\frak{g}$ and $\frak{g}^\ast$-module structures on $\frak{q}$ respectively which satisfy condition (a).  Since $D(\frak{g})=\frak{g}\oplus \frak{g}^\ast$ as a vector space, it follows that $\rho: D(\frak{g})\rightarrow \frak{gl}(\frak{q})$ satisfies conditions (i) and (ii) of Definition \ref{gqFLA8} iff $\phi: \frak{g}\rightarrow \frak{gl}(\frak{q})$ and $\psi: \frak{g}^\ast\rightarrow \frak{gl}(\frak{q})$ both satisfy conditions (i) and (ii) of Definition \ref{gqFLA8}.  This completes the proof. 
\end{proof}

\begin{proposition}
\label{Dg3}
Let $\frak{g}$ be a finite dimensional quasitriangular Lie bialgebra with r-matrix $r=\sum_i a_i\otimes b_i$.  Let $\varphi: \frak{g}\rightarrow \frak{gl}(V)$, $x\mapsto \varphi(x)$ be a represenation of $\frak{g}$ on $V$.  Define $\psi: \frak{g}^\ast\rightarrow \frak{gl}(V)$, $\xi\mapsto \psi(\xi)$ by
\begin{equation}
\label{Dg3e1}
\psi(\xi):=\sum_i \xi(a_i)\varphi(b_i),\hspace*{0.1in} \forall~\xi\in \frak{g}^\ast.
\end{equation}
Then $\psi$ is a representation of $\frak{g}^\ast$ on $V$.
\end{proposition}

\begin{proof}
We need to show that 
\begin{equation}
\label{Dg3e2}
\psi([\xi,\eta])=\psi(\xi)\psi(\eta)-\psi(\eta)\psi(\xi).
\end{equation}
We now expand the left side of (\ref{Dg3e2}):
\begin{align}
\nonumber
\psi([\xi,\eta])&=\sum_j [\xi,\eta](a_j)\varphi(b_j)\\
\nonumber
&=\sum_j (\xi\otimes \eta)((\delta r)(a_j))\varphi(b_j)\\
\label{Dg3e3}
&=\sum_{i,j}\xi([a_j,a_i])\eta(b_i)\varphi(b_j)+\sum_{i,j}\xi(a_i)\eta([a_j,b_i]))\varphi(b_j).
\end{align}
The right side of (\ref{Dg3e2}) expands as
\begin{align}
\nonumber
\psi(\xi)\psi(\eta)-\psi(\eta)\psi(\xi)&=\sum_{i,j}\xi(a_i)\eta(a_j)\varphi(b_i)\varphi(b_j)-\sum_{i,j}\eta(a_j)\xi(a_i)\varphi(b_j)\varphi(b_i)\\
\label{Dg3e4}
&=\sum_{i,j}\xi(a_i)\eta(a_j)\varphi([b_i,b_j]).
\end{align}
The CYBE can be rewritten as 
\begin{equation}
\label{Dg3e5}
\sum_{i,j} a_i\otimes a_j\otimes [b_i,b_j]=\sum_{i,j}[a_j,a_i]\otimes b_i\otimes b_j+\sum_{i,j} a_i\otimes [a_j,b_i]\otimes b_j.
\end{equation}
Applying $\xi\otimes \eta\otimes \varphi$ to both sides of (\ref{Dg3e5}) gives
\begin{equation}
\label{Dg3e6}
\sum_{i,j} \xi(a_i)\eta(a_j)\varphi([b_i,b_j])=\sum_{i,j}\xi([a_j,a_i])\eta(b_i)\varphi(b_j)+\sum_{i,j}\xi(a_i)\eta([a_j,b_i])\varphi(b_j).
\end{equation}
Equations (\ref{Dg3e3}), (\ref{Dg3e4}), and (\ref{Dg3e6}) imply
\begin{equation}
\nonumber
\psi(\xi)\psi(\eta)-\psi(\eta)\psi(\xi)=\psi([\xi,\eta]).
\end{equation}
This completes the proof.
\end{proof}

\begin{corollary}
\label{Dg4}
Let $\frak{g}$ be a finite dimensional quasitriangular Lie bialgebra with r-matrix $r=\sum_i a_i\otimes b_j$ and let $(\frak{q},\beta,\varphi)$ be a $\frak{g}$-quasi-Frobenius Lie algebra.  Define $\psi: \frak{g}^\ast \rightarrow \frak{gl}(\frak{q})$, $\xi\mapsto \psi(\xi)$ by 
\begin{equation}
\nonumber
\psi(\xi):=\sum_i \xi(a_i)\varphi(b_i),
\end{equation}
where $\varphi(b_i):=\varphi_{b_i}: \frak{q}\rightarrow \frak{q}$.  Then $(\frak{q},\beta,\psi)$ is a $\frak{g}^\ast$-quasi-Frobenius Lie algebra.
\end{corollary}
\begin{proof}
Immediate.
\end{proof}

\begin{proposition}
\label{Dg5}
Let $\frak{g}$ be a finite dimensional quasitriangular Lie bialgebra with r-matrix $r=\sum_i a_i\otimes b_i$.  Let $\varphi: \frak{g}\rightarrow \frak{gl}(V)$, $x\mapsto \varphi(x)$ be a represenation of $\frak{g}$ on $V$.  Define $\psi: \frak{g}^\ast\rightarrow \frak{gl}(V)$, $\xi\mapsto \psi(\xi)$ according to Proposition \ref{Dg3}.  Define $\rho: D(\frak{g})\rightarrow \frak{gl}(V)$, $a \mapsto \rho(a)$ by 
\begin{equation}
\rho(x+\xi):=\varphi(x)+\psi(\xi),\hspace*{0.1in}\forall~x\in \frak{g},~\xi\in \frak{g}^\ast.
\end{equation}
Then $\rho$ is a representation of $D(\frak{g})$ on $V$.
\end{proposition}
\begin{proof}
By Proposition \ref{Dg1}, it suffices to show that 
\begin{equation}
\label{Dg5e1}
\psi(ad^\ast_x\xi)-\varphi(ad^\ast_\xi x)=\varphi(x)\psi(\xi)-\psi(\xi)\varphi(x).
\end{equation}
We begin by expanding the left side of (\ref{Dg5e1}).  First,
\begin{align}
\nonumber
\psi(ad^\ast_x\xi)&=\sum_i (ad^\ast_x\xi)(a_i)\varphi(b_i)\\
\label{Dg5e2}
&=\sum_i \xi([a_i,x])\varphi(b_i).
\end{align}
By Proposition \ref{SSD11}, $\sum_i a_i\otimes b_i+\sum_i b_i\otimes a_i$ is invariant under the adjoint action of $\frak{g}$.  Hence, 
\begin{equation}
\label{Dg5e3}
\sum_i [a_i,x]\otimes b_i = \sum_i a_i\otimes [x,b_i]+\sum_i [x,b_i]\otimes a_i+\sum_i b_i\otimes [x,a_i].
\end{equation}
Equations (\ref{Dg5e2}) and (\ref{Dg5e3}) now imply
\begin{equation}
\label{Dg5e4}
\psi(ad^\ast_x\xi)=\sum_i \xi(a_i)\varphi([x,b_i])+\sum_i \xi([x,b_i])\varphi(a_i)+\sum_i \xi(b_i)\varphi([x,a_i]).
\end{equation}
Next, we note that
\begin{align}
\label{Dg5e5}
ad^\ast_\xi x=\sum_i \xi(b_i)[x,a_i]+\sum_i\xi([x,b_i])a_i.
\end{align}
From (\ref{Dg5e4}) and (\ref{Dg5e5}), we have
\begin{equation}
\label{Dg5e6}
\psi(ad^\ast_x\xi)-\varphi(ad^\ast_\xi x)=\sum_i \xi(a_i)\varphi([x,b_i]).
\end{equation}
For the right side of (\ref{Dg5e1}), we have
\begin{align}
\nonumber
\varphi(x)\psi(\xi)-\psi(\xi)\varphi(x)&=\sum_i \xi(a_i)\varphi(x)\varphi(b_i)- \sum_i \xi(a_i)\varphi(b_i)\varphi(x)\\
\nonumber
&=\sum_i \xi(a_i)\varphi([x,b_i])\\
\label{Dg5e7}
&=\psi(ad^\ast_x\xi)-\varphi(ad^\ast_\xi x),
\end{align}
where the last equality follows from (\ref{Dg5e6}).  This completes the proof.
\end{proof}

\begin{theorem}
\label{Dg6}
Let $\frak{g}$ be a finite dimensional quasitriangular Lie bialgebra.   Let $(\frak{q},\beta,\varphi)$ be any $\frak{g}$-quasi-Frobenius Lie algebra.  Then there exists a representation $\rho: D(\frak{g})\rightarrow \frak{gl}(\frak{q})$ such that $\rho|_{\frak{g}}=\varphi$ and $(\frak{q},\beta,\rho)$ is a $D(\frak{g})$-quasi-Frobenius Lie algebra.  
\end{theorem}

\begin{proof}
Let $r\in \frak{g}\otimes \frak{g}$ be the r-matrix associated to $\frak{g}$ and let $\psi: \frak{g}^\ast\rightarrow \frak{gl}(\frak{q})$ be the representation of $\frak{g}^\ast$ on $\frak{q}$ determined by $\varphi$ and $r$ according to Proposition \ref{Dg3}.  By Corollary \ref{Dg4}, $(\frak{q},\beta,\psi)$ is a $\frak{g}^\ast$-quasi-Frobenius Lie algebra.  Define $\rho: D(\frak{g})\rightarrow \frak{gl}(\frak{q})$ by 
\begin{equation}
\nonumber
\rho(x+\xi):=\varphi(x)+\psi(\xi),\hspace*{0.1in} \forall~x\in \frak{g},~\xi\in \frak{g}^\ast.
\end{equation}
By Proposition \ref{Dg5}, $\rho$ is a representation of $D(\frak{g})$ on $\frak{q}$.  Since $(\frak{q},\beta,\varphi)$ and $(\frak{q},\beta,\psi)$ are $\frak{g}$ and $\frak{g}^\ast$-quasi-Frobenius Lie algebras and $\rho|_{\frak{g}}=\varphi$ and $\rho|_{\frak{g}^\ast}=\psi$ (by definition), it follows that $(\frak{q},\beta,\rho)$ is a $D(\frak{g})$-quasi-Frobenius Lie algebra.
\end{proof}

\begin{corollary}
\label{Dg7}
Let $\frak{g}$ be any finite dimensional Lie algebra and let $(\frak{q},\beta,\varphi)$ be any $\frak{g}$-quasi-Frobenius Lie algbera.  Let $D(\frak{g})$ be the Drinfeld double of the Lie bialgebra $(\frak{g},\gamma)$ where $\gamma\equiv 0$.  Define $\rho: D(\frak{g})\rightarrow \frak{gl}(\frak{q})$ by $\rho(x+\xi)=\varphi(x)$ for all $x\in \frak{g}$, $\xi\in \frak{g}^\ast$.  Then $(\frak{q},\beta,\rho)$ is a $D(\frak{g})$-quasi-Frobenius Lie algebra.
\end{corollary}
\begin{proof}
 $(\frak{g},\gamma)$ is naturally a quasitriangular Lie bialgebra with r-matrix $r\equiv 0\in \frak{g}\otimes \frak{g}$.  Corollary \ref{Dg7} now follows as a special case of the proof of Theorem \ref{Dg6}. 
 \end{proof}

\noindent We conclude the paper with an example.
\begin{example}
Let $(\frak{q},\beta)$ be the 4-dimensional quasi-Frobenius Lie algebra from Example \ref{gqFLA13}.  For convenience, we recall its structure: $\frak{q}$ has basis $\{e_1,e_2,e_3,e_4\}$ with non-zero commutator relations given by
\begin{equation}
\nonumber
[e_1,e_2]=e_2,\hspace*{0.1in} [e_1,e_3]=e_3,\hspace*{0.1in} [e_1,e_4]=2e_4,\hspace*{0.1in} [e_2,e_3]=e_4,
\end{equation}
and the matrix representation of $\beta$ with respect to $\{e_1,e_2,e_3,e_4\}$ is 
\begin{equation}
\nonumber
\left(\beta_{ij}\right)=\left(\begin{array}{llll}
0 & 0 & 0 & 2\\
0 & 0 & 1 & 0\\
0 & -1 & 0 & 0\\
-2 & 0 & 0 & 0
\end{array}\right)
\end{equation}
Let $(\frak{g},\delta r)$ be the 2-dimensional triangular Lie bialgebra from Example \ref{trLBA} and \ref{DrinExample}.  Once again, we recall the structure for convenience.  $\frak{g}$ has basis $\{x,y\}$ with commutator relation $[x,y]=x$ and r-matrix $r=y\wedge x$.   Let $\{x^\ast,y^\ast\}$ denote the corresponding dual basis.  The commutator relations on $D(\frak{g})$ are 
\begin{align}
\nonumber
[x,y]&=x,\hspace*{0.1in} [x^\ast,y^\ast]=y^\ast,\hspace*{0.1in} [x,x^\ast]=-y^\ast,\hspace*{0.1in}[x,y^\ast]=0\\
\nonumber
& [y,x^\ast]=x^\ast+y,\hspace*{0.1in} [y,y^\ast]=-x
\end{align}
Let $\varphi: \frak{g}\rightarrow \frak{gl}(\frak{q})$ be the linear map defined by 
\begin{equation}
\nonumber
\varphi_x(e_1)=0,\hspace*{0.1in} \varphi_x(e_2)=0,\hspace*{0.1in} \varphi_x(e_3)=e_2,\hspace*{0.1in} \varphi_x(e_4)=0
\end{equation}
\begin{equation}
\nonumber
\varphi_y(e_1)=0,\hspace*{0.1in}\varphi_y(e_2)=-\frac{1}{2}e_2,\hspace*{0.1in}\varphi_y(e_3)=\frac{1}{2}e_3,\hspace*{0.1in}\varphi_y(e_4)=0.
\end{equation}
Consideration of Example \ref{gqFLA13} (or a direct calculation) shows that $(\frak{q},\beta,\varphi)$ is a $\frak{g}$-quasi-Frobenius Lie algebra.  By Theorem \ref{Dg6}, there exists a representation $\rho: D(\frak{g})\rightarrow \frak{gl}(\frak{q})$ such that $\rho|_{\frak{g}}=\varphi$ and $(\frak{q},\beta,\rho)$ is a $D(\frak{g})$-quasi-Frobenius Lie algebra.   We now compute $\rho$ explicitly.  From the proof of Theorem \ref{Dg6}, this amounts to computing the representation $\psi: \frak{g}^\ast\rightarrow \frak{gl}(\frak{q})$ which is determined by $\varphi$ and $r=y\wedge x$ according to Proposition \ref{Dg3}:
\begin{equation}
\nonumber
\psi_{x^\ast}=-\varphi_y,\hspace*{0.1in} \psi_{y^\ast}=\varphi_x.
\end{equation}
$\rho$ is then uniquely defined by $\rho|_{\frak{g}}=\varphi$ and $\rho|_{\frak{g}^\ast}=\psi$.
\end{example}

\end{document}